\documentclass[12pt,reqno]{amsart}
\usepackage[margin=1in]{geometry}
\usepackage{amsmath,amssymb,amsthm,graphicx,amsxtra, setspace}
\usepackage[utf8]{inputenc}
\usepackage{mathrsfs}
\usepackage{hyperref}
\usepackage{upgreek}
\usepackage{mathtools}
\usepackage[mathcal]{euscript}
\usepackage{xcolor}
\allowdisplaybreaks

\usepackage[pagewise]{lineno}

\newtheorem{theorem}{Theorem}[section]
\newtheorem{lemma}[theorem]{Lemma}

\newtheorem{assumption}[theorem]{Assumption}
\newtheorem{corollary}[theorem]{Corollary}

\newtheorem{remark}[theorem]{Remark}

\let\originalleft\left
\let\originalright\right
\renewcommand{\left}{\mathopen{}\mathclose\bgroup\originalleft}
\renewcommand{\right}{\aftergroup\egroup\originalright}

\renewcommand{\d}{\/\mathrm{d}\/}

\def\w{\textbf{W}^{\varepsilon}_{{\theta}^{\varepsilon}}}

\def\L{\mathbb{L}}
\def\A{\mathrm{A}}

\def\C{\mathrm{C}}
\def\f{\boldsymbol{f}}

\def\B{\mathrm{B}}
\def\D{\mathrm{D}}
\def\y{\boldsymbol{y}}

\def\z{\boldsymbol{z} }
\def\v{\boldsymbol{v}}
\def\w{\boldsymbol{w}}
\def\W{\mathrm{W}}

\def\N{\mathbb{N}}

\def\V{\mathbb{V}}
\def\wi{\widetilde}

\def\u{\mathrm{U}}

\def\u{\boldsymbol{u}}
\def\H{\mathbb{H}}

\newcommand{\R}{\mathbb{R}}

\renewcommand{\d}{\/\mathrm{d}\/}

\newcommand{\Addresses}{{
		\footnote{
			\noindent \textsuperscript{1,2}Department of Mathematics, Indian Institute of Technology Roorkee-IIT Roorkee,
			Haridwar Highway, Roorkee, Uttarakhand 247667, INDIA.\par\nopagebreak
			\noindent  \textit{e-mail:} \texttt{Manil T. Mohan: maniltmohan@ma.iitr.ac.in, maniltmohan@gmail.com.}
			
			\textit{e-mail:} \texttt{Kush Kinra: kkinra@ma.iitr.ac.in.}
			
			\noindent \textsuperscript{*}Corresponding author.
			
			\textit{Key words:} Global pullback attractor, random pullback attractor, upper semicontinuity, convective Brinkman-Forchheimer equations, unbounded domains.
			
			Mathematics Subject Classification (2020): Primary 35B41, 35Q35; Secondary 37L55, 37N10, 35R60.

}}}

\begin{document}
	
	\title[Pullback attractors of CBF and SCBF on $\mathbb{R}^d$]{Existence and upper semicontinuity of random pullback attractors for 2D and 3D non-autonomous stochastic convective Brinkman-Forchheimer equations on whole domain
		\Addresses}
	
	\author[K. Kinra and M. T. Mohan]
	{Kush Kinra\textsuperscript{1} and Manil T. Mohan\textsuperscript{2*}}

	\maketitle
	
	\begin{abstract}
		In this work, we analyze the long time behavior of 2D as well as 3D convective Brinkman-Forchheimer (CBF) equations and its stochastic counter part with non-autonomous deterministic forcing term in $\R^d$ $ (d=2, 3)$:
		$$\frac{\partial\boldsymbol{u}}{\partial t}-\mu \Delta\boldsymbol{u}+(\boldsymbol{u}\cdot\nabla)\boldsymbol{u}+\alpha\boldsymbol{u}+\beta|\boldsymbol{u}|^{r-1}\boldsymbol{u}+\nabla p=\boldsymbol{f},\quad \nabla\cdot\boldsymbol{u}=0,$$ where $r\geq1$. We prove the existence of a unique global pullback attractor for non-autonomous CBF equations, for $d=2$ with $r\geq1$, $d=3$ with $r>3$ and $d=r=3$ with $2\beta\mu\geq1$. For the same cases, we show the existence of a unique random pullback attractor for non-autonomous stochastic CBF equations with multiplicative white noise.  Finally, we establish the upper semicontinuity of the random pullback attractor, that is, the random pullback attractor converges towards the global pullback attractor when the noise intensity approaches to zero. Since we do not have compact Sobolev embeddings on unbounded domains, the pullback asymptotic compactness of the solution is proved by the method of energy equations given by Ball. For the case of Navier-Stokes equations defined on $\R^d$, such results are not available and the presence of Darcy term $\alpha\boldsymbol{u}$ helps us to establish the above mentioned results for CBF equations. 
	\end{abstract}

	\section{Introduction} \label{sec1}\setcounter{equation}{0}
	The present work is devoted to the asymptotic analysis of the mathematical model concerning convective Brinkman-Forchheimer (CBF) equations (or damped Navier-Stokes equations (NSE)) with non-autonomous forcing term (deterministic term) defined on whole domain. We establish the existence of a global pullback attractor for non-autonomous CBF equations and also prove the existence and upper semicontinuity of random pullback attractors for its stochastic counter part. Given $\tau\in\R$, we consider non-autonomous CBF equations in $\R^d\ (d=2,3)$, which describe the motion of incompressible fluid flows in a saturated porous medium, as (\cite{PAM}):
	\begin{equation}\label{1}
		\left\{
		\begin{aligned}
			\frac{\partial \u}{\partial t}-\mu \Delta\u+(\u\cdot\nabla)\u+\alpha\u+\beta|\u|^{r-1}\u+\nabla p&=\boldsymbol{f}, \ \ \ \ \ \ \  \text{ in }\  \R^d\times(\tau,\infty), \\ \nabla\cdot\u&=0, \ \ \ \ \ \ \ \  \text{ in } \ \ \R^d\times(\tau,\infty), \\
			\u(x,\tau)&=\u_0(x), \ \  \ x\in \R^d \ \text{ and }\ \tau\in\R,\\
			\u(x,\tau)&\to 0\ \ \ \ \ \ \ \ \  \text{ as }\ |x|\to\infty, 
		\end{aligned}
		\right.
	\end{equation}
	where $\u(x,t) :	\R^d\times(\tau,\infty)\to \R^d$ stands for the velocity field, $p(x,t):	\R^d\times(\tau,\infty)\to\R$ denotes the pressure field and $\f(x,t):	\R^d\times(\tau,\infty)\to \R^d$ is an external forcing. The constant $\mu>0$ stands for the \emph{Brinkman coefficient (effective viscosity)}, the constants $\alpha, \beta>0$ represent the \emph{Darcy (permeability of porous medium)} and \emph{Forchheimer coefficients}, respectively. The \emph{absorption exponent} $r\in[1,\infty)$ and  $r=3$ is called the \emph{critical exponent}. One can also consider the system \eqref{1} as a damped NSE (with $\alpha\u$ as linear damping and $\beta|\u|^{r-1}\u$ as nonlinear damping), since $\alpha=\beta=0$ recovers the classical NSE. Furthermore, the critical homogeneous CBF equations (\eqref{1} with $r=3$) and NSE have the same scaling (Proposition 1.1, \cite{HR}) only when $\alpha=0$ but no scale invariance property for other values of $\alpha$ and $r$. Therefore it is sometimes called NSE modified by an absorption term (\cite{SNA}) or the tamed NSE (\cite{MRXZ}). The CBF equations \eqref{1} is applied to flows when the velocities are sufficiently high and porosities are not too small, that is, when the Darcy law for a porous medium does not apply, that is, \emph{non-Darcy models} (see \cite{PAM} for more details).

	Let us now discuss some of the global solvability results available in the literature for 3D CBF equations. The Cauchy problem for CBF equations with $\alpha=0$ and $\beta=r$ is considered in \cite{ZCQJ}. The authors in \cite{ZCQJ} showed that the Cauchy problem for 3D CBF equations has global weak solutions, for any $r\geq 1$, global strong solutions, for any $r\geq 7/2$ and that the strong solution is unique, for any $7/2\leq r\leq 5$. An improvement for this result was made in \cite{ZZXW} by showing that the system \eqref{1} possesses global strong solutions, for any $r>3$ and the strong solution is unique, when $3<r\leq 5$. Later, the authors in \cite{YZ}  proved that the strong solution exists globally for $r\geq 3$, and they established two regularity criteria, for $1\leq r<3$. Moreover, for any $r\geq 1$, they proved that the strong solution is unique even among weak solutions.  The existence of unique weak as well as strong solutions for deterministic CBF equations defined on bounded and periodic domains can be obtained from \cite{SNA,FHR,HR,PAM,MTM}, etc. 
	
	Next, we list some literature on the  stochastic counterpart of the system \eqref{1} and related models. The existence of a unique strong solution  to stochastic tamed 3D NSE defined on the whole space is proved in \cite{MRXZ1}.  An improvement for their  result  for a slightly simplified system is obtained in \cite{ZBGD}. The local and global existence and uniqueness of solutions for a general class of deterministic and stochastic nonlinear evolution equations with coefficients satisfying some local monotonicity and generalized coercivity conditions is proved in \cite{WLMR}. The existence and uniqueness of strong solutions for a large class of stochastic partial differential equations (SPDEs), where the coefficients satisfy the local monotonicity and Lyapunov condition, is showed \cite{WL},  and the author provided the stochastic tamed 3D NSE as an example. The global solvability of 3D NSE in the whole space with a Brinkman-Forchheimer type term subject to an anisotropic viscosity and a random perturbation of multiplicative type is described in the work \cite{HBAM}. The existence of strong solutions (in the probabilistic sense) and martingale solutions for stochastic CBF equations on bounded domains are proved in \cite{MTM1} and \cite{LHGH1}, respectively. The existence of a unique pathwise strong solution to 3D stochastic NSE is a well known open problem, and the same is open for 3D stochastic CBF equations with  $r\in[1,3)$. Therefore, we restrict ourselves to $d=2$ with $r\geq1$ for any $\mu, \beta>0$, $d=3$ with $r>3$ for any $\mu, \beta>0$ and $d=r=3$ for $2\beta\mu\geq1$ (for more details, see \cite{MTM1}).
	
	For the dynamical systems generated by fluid flow models, the theory of attractors have great importance in the  study of asymptotic behaviors and their qualitative properties. For a comprehensive study on the theory of global attractors for autonomous dynamical systems, the interested readers are referred to see \cite{JCR,R.Temam} etc. The concept of global pullback attractors for non-autonomous dynamical systems using pullback asymptotically compactness was first introduced in \cite{CLR}, which is an extension of the similar concepts in the autonomous framework. Moreover, the authors in \cite{CLR} applied this theory to 2D non-autonomous NSE on some unbounded domains. There is a good number of literature available on the theory of attractors for autonomous (see \cite{KM3,MTM2,Rosa}, etc) and non-autonomous dynamical systems (see \cite{CLR1,CLR2,LW,ZXZZ} etc and the references therein).
	
	It is well-investigated in the literature that a large class of SPDEs generate random dynamical systems (RDS) (cf. \cite{Arnold}). The fundamental theory of attractors for infinite dimensional RDS was established in \cite{BCF,CDF,CF}, etc and several authors applied this theory to various physically relevant SPDEs, see for example \cite{FY,FS,KM3,LG,You}, etc and the references therein. In particular, the work \cite{KM}, etc discussed random attractors for autonomous stochastic CBF equations on some unbounded domains like Poincar\'e domains. Furthermore, the existence of weak pullback mean random attractors for non-autonomous stochastic CBF equations was proved recently in \cite{KM4}. To the best of authors' knowledge,  results on the existence of pullback random attractors for non-autonomous stochastic NSE or CBF equations on whole domains is not available in literature. In this work, we resolve this problem for stochastic CBF equations  with the help of Darcy term $\alpha\u$ and the abstract theory developed in  \cite{SandN_Wang}. Necessary and sufficient criteria for the existence of periodic random pullback attractors for non-autonomous non-compact RDS is presented in \cite{SandN_Wang}. Moreover, the author  proved the existence of a unique random pullback attractor for non-autonomous stochastic Reaction-Diffusion equations on whole domains as an application of this theory. Using this theory, the periodic pullback random attractor for stochastic NSE with time dependent deterministic term on Poincar\'e domains is established in \cite{PeriodicWang}. For more applications of this theory, the interested readers are referred to see  \cite{CZ,GGW,HZ,Wang,RB1,ZZ}, etc.
	
	In this work, we consider the stochastic CBF equations perturbed by multiplicative white noise of  the form $\varepsilon\u\circ\frac{\d \W(t)}{\d t}$, where  $\varepsilon>0$, $\W(\cdot)$ be the two-sided Wiener process on the probability space $(\Omega,\mathscr{F},\mathbb{P})$ and $\circ$ means that the stochastic integral is understood in the  sense of Stratonovich (see section \ref{sec4}). The \emph{stability of the random pullback attractors under stochastic perturbations} is one of the important properties of the random pullback attractors, that is, the convergence of random pullback attractors	$\mathscr{A}_{\varepsilon}(\tau,\omega)$ towards the deterministic global pullback attractor $\mathscr{A}_0(\tau)$, when the noise intensity $\varepsilon$ approaches to zero. For a metric space $(\mathbb{X},d)$, there are two types of convergences,
	\vskip 1mm
	\noindent
	\begin{align*}
		\lim\limits_{\varepsilon\to0^{+}} \text{dist}_{\mathbb{X}}\left(\mathscr{A}_{\varepsilon}(\tau,\omega),\mathscr{A}_0(\tau)\right)=0\ \  \text{ and }\ \  \lim\limits_{\varepsilon\to0^{+}} \text{dist}_{\mathbb{X}}\left(\mathscr{A}_0(\tau),\mathscr{A}_{\varepsilon}(\tau,\omega)\right)=0,
	\end{align*}
	\vskip 1mm
	\noindent
	which is called \emph{upper semicontinuity} and \emph{lower semicontinuity}, respectively. Here, $\text{dist}_{\mathbb{X}}(\cdot,\cdot)$ denotes the Hausdorff semi-distance between two non-empty subsets of $\mathbb{X}$, that is, for non-empty sets $A,B\subset \mathbb{X}$ $$\text{dist}_{\mathbb{X}}(A,B)=\sup_{a\in A}\inf_{b\in B} d(a,b).$$ 
	For lower semicontinuity, we need a more detailed study either on the equi-attraction of the family of the random attractors of perturbed systems or on the structure of the deterministic attractors. Therefore we prove the results of upper semicontinuity only in this work and the lower semicontinuity will be addressed in a future work. Abstract results on upper semicontinuity of the random attractors for autonomous compact RDS was introduced in \cite{CLR_USC}. Later, the author in \cite{Wang} presented the theory of upper semicontinuity for autonomous non-compact RDS also. For upper semicontinuity of random attractors for autonomous stochastic models on bounded and unbounded domains, see \cite{XJXD,KM1,LF,Wang,ZW}, etc and the references therein. In particular, upper  and lower semicontinuity results for stochastic NSE and stochastic CBF equations on periodic domains is presented in \cite{HCPEK} and \cite{KM5}, respectively. In \cite{non-autoUpperWang}, the author extended the results on upper semicontinuity of the works \cite{CLR_USC,Wang} to  non-autonomous RDS (compact as well as non-compact), and it has been applied in the works  \cite{CZ,RB1,ZZ}, etc. 
	
	We fix  $r\in[1,\infty)$ with any $\mu,\beta>0$ in 2D, and $r\in(3,\infty)$ with any $\mu,\beta>0$  and $r=3$ with $2\mu\beta\geq1$ in 3D. For  2D and 3D non-autonomous CBF equations on whole domains, the major  three aims of this article are as follows:   For $\f\in\mathrm{L}^2(0,T;\V')$, 
	\begin{itemize}
		\item [(i)] the existence of a unique \emph{global pullback attractor},
		\item[(ii)] the existence of a unique \emph{random pullback attractor},
		\item[(iii)] \emph{upper semicontinuity of random attractors}.
	\end{itemize}
	As mentioned earlier, the results which we prove in this article for CBF equations are not available for NSE  on whole domain. That is, the existence of global pullback attractors for NSE and, existence and upper semicontinuity of random pullback attractors for stochastic NSE on whole domains is still an open problem. It is worth to mention here that linear damping term or Darcy term (that is, $\alpha\u$ with $\alpha>0$) plays a pivotal role in obtaining the above mentioned results on whole domains.

	Let us now discuss the parametric spaces $\Omega_1$ and $\Omega_2$, which we need to generate cocycles for the solution of non-autonomous deterministic as well as stochastic CBF equations. Here, $\Omega_1$ is a non-empty set needed to deal with time-dependent forcing term $\f$ and $\Omega_2$ is a probability space which manages stochastic term. In particular, we examine how to choose the parametric space $\Omega_1$ for the time-dependent forcing term $\f$ so that the solution operators of equations \eqref{CBF} and \eqref{SCBF} can be formulated into the setting of cocycles.  It is presented in \cite{SandN_Wang} that we can choose the space $\Omega_1$ by at least two ways. We may take $\Omega_1$ either as the set of all translations of the deterministic terms or as the set of all initial times, that is, $\Omega_1=\R$. It is demonstrated in \cite{SandN_Wang} that we get the same results for the choices of $\Omega_1$ given above. In this article, we will choose $\Omega_1=\R$ and $\Omega_2=\Omega$, where $\Omega$ is given by \eqref{Omega}. 
	
	When one considers evolution equations on unbounded domains and tries to demonstrate the existence of attractors, both in deterministic and stochastic settings, the major drawback is non-compactness of Sobolev embeddings. To avoid this difficulty, one can use the asymptotic compactness method. One way to obtain the asymptotic compactness is the use of  weak convergence and energy equation (cf. \cite{KM,MTM2,PeriodicWang} etc). An another way to achieve this by approximating unbounded domains by bounded domains and using the fact that the approximation error of the norm of solutions is arbitrarily small for large time (uniformly)  or sometimes it is called uniform tail estimate for the solutions (see \cite{Wang,SandN_Wang,ZW} etc). We point out here that we are able to prove uniform tail estimate for the solutions only if $r\in[2,\infty),$ where nonlinear damping term $\beta|\u|^{r-1}\u$ (for $\beta>0$) plays key role, and $\f\in\mathrm{L}^2(0,T;\H)$ (see Lemma \ref{largeradius}). Therefore we use the method of energy equation given in \cite{Ball} to prove the asymptotic compactness which gives the results for all $r\in[1,\infty)$ and $\f\in\mathrm{L}^2(0,T;\V')$ (see Theorem \ref{D-asymp}, Lemmas \ref{Asymptotic_v} and \ref{precompact}).

	The rest of the paper is organized as follows. In the next section, we define some functional spaces which are needed in the further analysis. Moreover, linear, bilinear and nonlinear operators along with their important properties are also defined in the same section. In section \ref{sec3}, we consider the non-autonomous deterministic CBF equations \eqref{CBF} and establish the existence of a unique global pullback attractor (Theorem \ref{Main_T1}) by proving asymptotic compactness using  the method of energy equation given in \cite{Ball} (Theorem \ref{D-asymp}). In section \ref{sec4}, we consider the non-autonomous stochastic CBF equations perturbed by multiplicative noise \eqref{SCBF}. In order to deal with the stochastic term, we convert the system \eqref{SCBF} into the system \eqref{CCBF} (a system which is deterministic for each $\omega\in\Omega$) with the help of a transformation given by \eqref{Trans2}. Also, we define the non-autonomous random dynamical system (NRDS) for the system \eqref{SCBF} in the same section. We establish the existence of a random pullback attractor for the system \eqref{SCBF} in section \ref{sec5}. Initially, we prove an absorbing set (Lemma \ref{LemmaUe}) and asymptotic compactness (Lemma \ref{Asymptotic_v}) for the system \eqref{CCBF}. Later, using the transformation given in \eqref{Trans2}, we prove the existence of a unique random pullback attractor for the system \eqref{SCBF} (Theorem \ref{PullbackAttractor}). In section \ref{sec6}, we demonstrate the upper semicontinuity of random pullback attractors for the system \eqref{SCBF} using the abstract theory given in \cite{non-autoUpperWang} (Theorem \ref{Main_T3}). In the Appendix \ref{sec7}, we discuss the uniform tail estimate for the solution of stochastic CBF equations \eqref{SCBF}.

	\section{Mathematical Formulation}\label{sec2}\setcounter{equation}{0}
	This section is devoted for providing the necessary function spaces needed to obtain the results of this work.
	\subsection{Function spaces} 
	We define the space $$\mathcal{V}:=\{\u\in\C_0^{\infty}(\mathbb{R}^d;\mathbb{R}^d):\nabla\cdot\u=0\},$$ where $\C_0^{\infty}(\mathbb{R}^d;\mathbb{R}^d)$ denotes the space of all infinitely differentiable functions  ($\mathbb{R}^d$-valued) with compact support in $\mathbb{R}^d$. Let $\H$, $\V$ and $\wi\L^p$ denote the completion of $\mathcal{V}$ in 	$\mathrm{L}^2(\mathbb{R}^d;\mathbb{R}^d)$, $\mathrm{H}^1(\mathbb{R}^d;\mathbb{R}^d)$ and $\mathrm{L}^p(\mathbb{R}^d;\mathbb{R}^d)$, $p\in(2,\infty)$,  norms, respectively. The space $\H$ is endowed with the norm $\|\u\|_{\H}^2:=\int_{\R^d}|\u(x)|^2\d x,$ the norm on the space $\widetilde{\L}^{p}$ is defined by $\|\u\|_{\wi \L^p}^2:=\int_{\R^d}|\u(x)|^p\d x,$ for $p\in(2,\infty)$ and the norm on the space $\V$ is given by $\|\u\|^2_{\V}=\int_{\R^d}|\u(x)|^2\d x+\int_{\R^d}|\nabla\u(x)|^2\d x.$ The inner product in the Hilbert space $\H$ is denoted by $( \cdot, \cdot)$. The duality pairing between the spaces $\V$ and $\V'$, and $\widetilde{\L}^p$ and its dual $\widetilde{\L}^{\frac{p}{p-1}}$ is represented by $\langle\cdot,\cdot\rangle.$ It should be noted that $\H$ can be identified with its own dual $\H'$. We endow the space $\V\cap\widetilde{\L}^{p}$ with the norm $\|\u\|_{\V}+\|\u\|_{\widetilde{\L}^{p}},$ for $\u\in\V\cap\widetilde{\L}^p$ and its dual $\V'+\widetilde{\L}^{p'}$ with the norm $$\inf\left\{\max\left(\|\v_1\|_{\V'},\|\v_1\|_{\widetilde{\L}^{p'}}\right):\v=\v_1+\v_2, \ \v_1\in\V', \ \v_2\in\widetilde{\L}^{p'}\right\}.$$ 

	\subsection{Linear operator}
	Let $\mathcal{P}: \L^2(\R^d) \to\H$ be the Helmholtz-Hodge (or Leray) projection. Note that the projection opertaor $\mathcal{P}$ can be expressed in terms of the Riesz transform (cf. \cite{MTSS}).  We define the Stokes operator
	\begin{equation*}
		\A\u:=-\mathcal{P}\Delta\u,\;\u\in\D(\A):=\V\cap\H^{2}(\R^d).
	\end{equation*}
	It should be noted that $\mathcal{P}$ and $\Delta$ commutes, that is, $\mathcal{P}\Delta=\Delta\mathcal{P}$. 
	\subsection{Bilinear operator}
	Let us define the \emph{trilinear form} $b(\cdot,\cdot,\cdot):\V\times\V\times\V\to\R$ by $$b(\u,\v,\w)=\int_{\R^d}(\u(x)\cdot\nabla)\v(x)\cdot\w(x)\d x=\sum_{i,j=1}^n\int_{\R^d}\u_i(x)\frac{\partial \v_j(x)}{\partial x_i}\w_j(x)\d x.$$ If $\u, \v$ are such that the linear map $b(\u, \v, \cdot) $ is continuous on $\V$, the corresponding element of $\V'$ is denoted by $\B(\u, \v)$. We also denote $\B(\u) = \B(\u, \u)=\mathcal{P}[(\u\cdot\nabla)\u]$.	An integration by parts gives 
	\begin{equation}\label{b0}
		\left\{
		\begin{aligned}
			b(\u,\v,\v) &= 0,\ \text{ for all }\ \u,\v \in\V,\\
			b(\u,\v,\w) &=  -b(\u,\w,\v),\ \text{ for all }\ \u,\v,\w\in \V.
		\end{aligned}
		\right.\end{equation}
	\begin{remark}
		The following estimates on the trilinear form $b(\cdot,\cdot,\cdot)$ are used in the sequel (see Chapter 2, section 2.3, \cite{Temam1}): For all $\u, \v, \w\in \V$, 
			\begin{align}
				|b(\u,\v,\w)|&\leq
				C\|\u\|^{1/2}_{\H}\|\nabla\u\|^{1/2}_{\H}\|\nabla\v\|_{\H}\|\w\|^{1/2}_{\H}\|\nabla\w\|^{1/2}_{\H},\quad\text{ for }n=2,\label{b1}
			\end{align}
		and 
			\begin{align}
				|b(\u,\v,\w)|&\leq
				C\|\u\|^{1/4}_{\H}\|\nabla\u\|^{3/4}_{\H}\|\nabla\v\|_{\H}\|\w\|^{1/4}_{\H}\|\nabla\w\|^{3/4}_{\H},\quad\text{ for }n=3. \label{b3}
			\end{align}
	\end{remark}
	\begin{remark}
		Note that $\langle\B(\u,\u-\v),\u-\v\rangle=0$ , which  implies that
		\begin{equation}\label{441}
			\begin{aligned}
				\langle \B(\u)-\B(\v),\u-\v\rangle =\langle\B(\u-\v,\v),\u-\v\rangle=-\langle\B(\u-\v,\u-\v),\v\rangle.
			\end{aligned}
		\end{equation} 
	\end{remark}
	\subsection{Nonlinear operator}
	Let us consider the nonlinear operator $\mathcal{C}(\u):=\mathcal{P}(|\u|^{r-1}\u),$ for $\u\in\V\cap\wi\L^{r+1}$. It is obvious that $\langle\mathcal{C}(\u),\u\rangle =\|\u\|_{\widetilde{\L}^{r+1}}^{r+1}$. Moreover, the map $\mathcal{C}(\cdot):\V\cap\widetilde{\L}^{r+1}\to\V'+\widetilde{\L}^{\frac{r+1}{r}}$.  For any $r\in [1, \infty)$ and $\u, \v \in \V\cap\widetilde{\L}^{r+1}$, we have (see subsection 2.4, \cite{MTM1})
	\begin{align}\label{MO_c}
		\langle\mathcal{C}(\u)-\mathcal{C}(\v),\u-\v\rangle\geq\frac{1}{2}\||\u|^{\frac{r-1}{2}}(\u-\v)\|_{\H}^2+\frac{1}{2}\||\v|^{\frac{r-1}{2}}(\u-\v)\|_{\H}^2 \geq 0.
	\end{align}
	\begin{lemma}[Interpolation inequality] \label{Interpolation}
		Assume $1\leq s_1\leq s\leq s_2\leq \infty$, $\ell\in(0,1)$ such that $\frac{1}{s}=\frac{\ell}{s_1}+\frac{1-\ell}{s_2}$ and $\u\in\L^{s_1}(\R^d)\cap\L^{s_2}(\R^d)$, then
		\begin{align}\label{211}
			\|\u\|_{\L^s(\R^d)}\leq\|\u\|_{\L^{s_1}(\R^d)}^{\ell}\|\u\|_{\L^{s_2}(\R^d)}^{1-\ell}. 
		\end{align}
	\end{lemma}
	
	\section{Global pullback attractor for non-autonomous CBF equations }\label{sec3}\setcounter{equation}{0} This section is dedicated for establishing the global pullback attractor for non-autonomous CBF equations. 
	\subsection{Abstract formulation}
	Taking orthogonal projection $\mathcal{P}$ to the system \eqref{1}, we get 
	\begin{equation}\label{CBF}
		\left\{
		\begin{aligned}
			\frac{\d\u(t)}{\d t}+\mu \A\u(t)+\B(\u(t))+\alpha\u(t) +\beta\mathcal{C}(\u(t))&=\f(t) , \quad \text{ in } \R^d\times(\tau,\infty), \\ 
			\u(x,\tau)&=\u_{0}(x),\quad x\in \R^d \text{ and } \tau\in\R,
		\end{aligned}
		\right.
	\end{equation} 
	in $\V'+\widetilde{\L}^{\frac{r+1}{r}}$, where $\u_0\in\H$ and $\f\in\mathrm{L}^2_{\mathrm{loc}}(\R;\V')$. The existence and uniqueness of solution to the system \eqref{CBF} is proved in \cite{MTM}. Moreover, the solution $\u(\cdot)$ belongs to $\mathrm{C}([\tau,+\infty);\H)\cap\mathrm{L}^2_{\mathrm{loc}}(\tau,+\infty;\V)\cap\mathrm{L}^{r+1}_{\mathrm{loc}}(\tau,+\infty;\widetilde{\L}^{r+1})$. 
	
	In order to define the non-autonomous dynamical system generated by \eqref{CBF}, we consider $\{\theta_t\}_{t\in\R}$ be a family of shift operator on $\R,$ which is given by
	\begin{align}\label{theta}
		\theta_t \tau=\tau+t, \ \ \ \text{ for all } \tau\in\R.
	\end{align} and define
	\begin{align}\label{phi}
		\Phi_0(t,\tau,\u_0)=\u(t+\tau;\tau,\u_{0}),\qquad\tau\in\R, \ t\geq0,\ \u_0\in\H.
	\end{align}
	Since, $\u(\cdot)$ is the unique solution to the system \eqref{CBF}, it implies that
	\begin{align}\label{phi1}
		\Phi_0(t+s,\tau,\u_0)=\Phi_0(t,s+\tau,\Phi_0(s,\tau,\u_0)),\qquad\tau\in\R, \ t, s\geq0,\ \u_0\in\H.
	\end{align}
	Also, it is easy to prove by a standard method that for all $\tau\in\R, t\geq0$ the mapping $\Phi_0(t,\tau,\cdot)$ defined in \eqref{phi} is continuous from $\H$ into itself. Therefore, the mapping $\Phi_0$ given by \eqref{phi} is continuous $\theta$-cocycle on $\H$. 
	Throughout this section, we assume that external forcing term also satisfies the following:
	\begin{align}\label{forcing}
		\int_{-\infty}^{\tau}e^{\alpha \xi}\|\f(\cdot,\xi)\|^2_{\V'}\d\xi<+\infty, \text{ for all } \tau\in \R.
	\end{align}
	Let $E\subseteq\H$ and denote by 
	$$\|E\|_{\H}=\sup_{x\in E}\|x\|_{\H}.$$ Assume that $D=\{D(\tau)\}_{\tau\in\R}$ is a family of non-empty subsets of $\H$ satisfying
	\begin{align}\label{D_0}
		\lim_{t\to-\infty}e^{\alpha t}\|D(\tau+t)\|^2_{\H}=0,
	\end{align}
	where $\alpha>0$ is the Darcy coefficient. Also, let $\mathfrak{D}_0$ denote the set of all families of subsets of $\H$ satisfying \eqref{D_0}, that is,
	\begin{align}\label{D_01}
		\mathfrak{D}_0=\{D=\{D(\tau)\}_{\tau\in\R}:D \text{ satisfying } \eqref{D_0}\}.
	\end{align} 
	\begin{theorem}\label{D-abH}
		For $d=2$ with $r\geq1$, $d=3$ with $r>3$ and $d=r=3$ with $2\beta\mu\geq1$, assume that $\f\in \mathrm{L}^2_{\mathrm{loc}}(\mathbb{R};\V')$ satisfies \eqref{forcing}. Then for every $\tau\in\R$ and $K=\{K(\tau)\}_{\tau\in\R}\in\mathfrak{D}_{0}$, there exists $\mathscr{T}=\mathscr{T}(\tau,K)>0$ such that for all $t\geq \mathscr{T}$, 
		\begin{align}\label{S1*}
			\|\u(\tau;\tau-t,\u^0)\|^2_{\H}\leq 1 + \frac{e^{-\alpha \tau}}{\min\{\mu,\alpha\}}\int_{-\infty}^{\tau}e^{\alpha \xi}\|\f(\cdot,\xi)\|^2_{\V'}\d\xi,
		\end{align}
		where $\u^0\in K(\tau-t)$.
	\end{theorem}
	\begin{proof}
		Let us take $\u(\cdot)=\u(\cdot;\tau-t,\u^0)$. 	From the first equation of \eqref{CBF}, we get
		\begin{align}
			\frac{1}{2}&\frac{\d}{\d s}\|\u(s)\|^2_{\H}+\mu\|\nabla\u(s)\|^2_{\H}+\alpha\|\u(s)\|^2_{\H}+\beta\|\u(s)\|^{r+1}_{\wi \L^{r+1}}\nonumber\\&=\left\langle\f(s),\u(s)\right\rangle\leq\|\f(s)\|_{\V'}\|\u(s)\|_{\V}\nonumber\\&\leq\frac{1}{2\min\{\mu,\alpha\}}\|\f(s)\|^2_{\V'}+\frac{\min\{\mu,\alpha\}}{2}\|\u(s)\|^2_{\V},\label{S*1}
		\end{align}
		so that 
		\begin{align}
			&\frac{\d}{\d s}\|\u(s)\|^2_{\H}+\alpha\|\u(s)\|^2_{\H}\leq\frac{1}{\min\{\mu,\alpha\}}\|\f(s)\|^2_{\V'},\label{S*3}
		\end{align}
		for a.e. $[\tau-t,\tau-t+T]$ with $T>0$ and it follows from the variation of constants formula that
		\begin{align}\label{S*4}
			e^{\alpha \tau}\|\u(\tau;\tau-t,\u^0)\|^2_{\H} &\leq e^{\alpha (\tau-t)}\|\u^0\|^2_{\H}+ \frac{1}{\min\{\mu,\alpha\}}\int_{\tau-t}^{\tau}e^{\alpha \xi}\|\f(\cdot,\xi)\|^2_{\V'}\d\xi.
		\end{align}
		Since $\u^0\in K(\tau-t)$, there exists $\mathscr{T}=\mathscr{T}(\tau,K)$ such that for all $t\geq \mathscr{T},$
		\begin{align*}
			e^{-\alpha t}\|\u^0\|^2_{\H}\leq 1,
		\end{align*}
		which conclude the proof together with \eqref{S*4}.
	\end{proof}
	The next result will be used to prove asymptotically compactness of $\Phi_0$. A proof of the following Theorem can be obtained similar to that of  Theorem \ref{D-abH}, and hence we omit here.
	\begin{theorem}\label{D-abH2}
		For $d=2$ with $r\geq1$, $d=3$ with $r>3$ and $d=r=3$ with $2\beta\mu\geq1$, assume that $\f\in \mathrm{L}^2_{\mathrm{loc}}(\mathbb{R};\V')$ satisfies \eqref{forcing}. Then for every $\tau\in\R$ and $K=\{K(\tau)\}_{\tau\in\R}\in\mathfrak{D}_{0}$, there exists $\mathscr{T}=\mathscr{T}(\tau,K)>0$ such that for every $k\geq0$ and for all $t\geq \mathscr{T}+k$, 
		\begin{align}\label{S*2}
			\|\u(\tau-k;\tau-t,\u^0)\|^2_{\H}\leq e^{\alpha k} + \frac{e^{-\alpha (\tau-k)}}{\min\{\mu,\alpha\}}\int_{-\infty}^{\tau}e^{\alpha \xi}\|\f(\cdot,\xi)\|^2_{\V'}\d\xi,
		\end{align}
		where $\u^0\in K(\tau-t)$.
	\end{theorem}
	\begin{lemma}\label{D-convege}
		For $d=2$ with $r\geq1$, $d=3$ with $r>3$ and $d=r=3$ with $2\beta\mu\geq1$, assume that $\f\in\mathrm{L}^2_{\mathrm{loc}}(\R;\V')$ and $\{\u_n^0\}_{n\in\N}\subset\H$ be a sequence converging weakly to $\u^{0}\in\H$, in $\H$. Then
		\begin{align*}
			\u(t;\tau,\u_n^0)\xrightharpoonup{w} \u(t;\tau,\u^{0}) \ \text{ in }\  \H, \quad\text{ for all } t\geq\tau,\quad \tau\in\R,
		\end{align*}
		\begin{align*}
			\u(\cdot\ ;\tau,\u_n^0)\xrightharpoonup{w} \u(\cdot\ ;\tau,\u^{0}) \ \text{ in }\  \mathrm{L}^2(\tau,\tau+T;\V) \text{ and }\  \mathrm{L}^{r+1}(\tau,\tau+T;\wi\L^{r+1}), \text{ for every } T>0.
		\end{align*}
	\end{lemma}
	\begin{proof}
		The proof of this lemma is analogous to Lemma 4.1 in \cite{MTM2} (in particular, we need to prove it for $\Phi_0$ given in \eqref{phi}) and hence we omit it here.
	\end{proof}
	\begin{theorem}\label{D-asymp}
		For $d=2$ with $r\geq1$, $d=3$ with $r>3$ and $d=r=3$ with $2\beta\mu\geq1$, assume that $\f\in\mathrm{L}^2_{\mathrm{loc}}(\R;\V')$ and satisfies \eqref{forcing}. Then, $\Phi_0$ is $\mathfrak{D}_{0}$-pullback asymptotically compact in $\H$.
	\end{theorem}
	\begin{proof}
		To prove the required result, we need to show that for every $\tau\in\R, K=\{K(\tau)\}_{\tau\in\R}\in\mathfrak{D}_{0}$, and $t_n\to\infty, \u_n^0\in K(\tau-t_n),$ the sequence $\Phi_0(t_n,\tau-t_n,\u_n^0)$ has a convergent subsequence in $\H$. From Theorem \ref{D-abH}, it follows  that there exists $\mathscr{T}=\mathscr{T}(\tau,K)>0$ such that for all $t\geq \mathscr{T}$,
		\begin{align}\label{S1}
			\|\u(\tau;\tau-t,\u^0)\|^2_{\H} \leq 1+ \frac{e^{-\alpha \tau}}{\min\{\mu,\alpha\}}\int_{-\infty}^{\tau}e^{\alpha \xi}\|\f(\cdot,\xi)\|^2_{\V'}\d\xi,
		\end{align}
		where $\u^0\in K(\tau-t).$ Since $t_n\to \infty$, there exists $N_0\in\N$ such that $t_n\geq \mathscr{T},$ for all $n\geq N_0$. Since $\u_n^0\in K(\tau-t_n)$,  then \eqref{S1} gives that 
		\begin{align}\label{S2}
			\|\u(\tau;\tau-t_n,\u_n^0)\|^2_{\H} \leq 1+ \frac{e^{-\alpha \tau}}{\min\{\mu,\alpha\}}\int_{-\infty}^{\tau}e^{\alpha \xi}\|\f(\cdot,\xi)\|^2_{\V'}\d\xi,
		\end{align}
		for all $n\geq N_0.$ From \eqref{S2}, it is clear that $\u(\tau;\tau-t_n,\u_n^0)$ is bounded in $\H$ for all $n\geq N_0$, therefore there exists a subsequence (denote here by same notation) and $\tilde{\u}\in \H$ such that 
		\begin{align}\label{S3}
			\u(\tau;\tau-t_n,\u_n^0)\xrightharpoonup{w}\tilde{\u} \ \text{ in }\  \H,
		\end{align}
		and the above weak convergence also implies that
		\begin{align}\label{S4}
			\|\tilde{\u}\|_{\H}\leq\liminf_{n\to\infty}\|\u(\tau;\tau-t_n,\u_n^0)\|_{\H}.
		\end{align}
		In order to show that the convergence in \eqref{S3} is in fact a strong convergence, we only need to prove 
		\begin{align}\label{S5}
			\|\tilde{\u}\|_{\H}\geq\limsup_{n\to\infty}\|\u(\tau;\tau-t_n,\u_n^0)\|_{\H}.
		\end{align}
		To prove \eqref{S5}, we apply the idea of energy equations presented in \cite{Ball}.  For a given $k\in \N$, we have 
		\begin{align}\label{S6}
			\u(\tau;\tau-t_n,\u_n^0)=\u(\tau;\tau-k,\u(\tau-k;\tau-t_n,\u_n^0)).
		\end{align} 
		For each $k$, let $N_k$ be sufficiently large such that $t_n\geq \mathscr{T}+k$ for all $n\geq N_k$. From Theorem \ref{D-abH2}, we find 
		\begin{align}\label{S_6}
			\|\u(\tau-k;\tau-t_n,\u_n^0)\|^2_{\H} \leq e^{\alpha k}+ \frac{e^{-\alpha (\tau-k)}}{\min\{\mu,\alpha\}}\int_{-\infty}^{\tau}e^{\alpha \xi}\|\f(\cdot,\xi)\|^2_{\V'}\d\xi,
		\end{align}
		for $k\geq N_k.$ For each fixed $k\in \N$, it is obvious from \eqref{S_6} that $\u(\tau-k;\tau-t_n,\u_n^0)$ is a bounded sequence in $\H$. By the diagonal process, there exists a subsequence (for convenience, we use the same) and $\tilde{\u}_k\in \H$ for each $k\in\N$ such that 
		\begin{align}\label{S7}
			\u(\tau-k;\tau-t_n,\u_n^0)\xrightharpoonup{w}\tilde{\u}_k \ \text{ in } \ \H.
		\end{align}
		It follows from \eqref{S6},\eqref{S7} and Lemma \ref{D-convege} that for $k\in\N$,
		\begin{align}\label{S8}
			\u(\tau;\tau-t_n,\u_n^0)\xrightharpoonup{w}\u(\tau;\tau-k,\tilde{\u}_{k}) \ \text{ in } \ \H,
		\end{align}
		\begin{align}\label{S9}
			\u(\cdot\ ;\tau-k,\u(\tau-k;\tau-t_n,\tilde{\u}_{k}))\xrightharpoonup{w}\u(\cdot\ ;\tau-k,\tilde{\u}_{k}) \ \text{ in } \ \mathrm{L}^2(\tau-k,\tau;\V).
		\end{align}
		and
		\begin{align}\label{S9'}
			\u(\cdot\ ;\tau-k,\u(\tau-k;\tau-t_n,\tilde{\u}_{k}))\xrightharpoonup{w}\u(\cdot\ ;\tau-k,\tilde{\u}_{k}) \ \text{ in } \ \mathrm{L}^{r+1}(\tau-k,\tau;\wi\L^{r+1}).
		\end{align}
		Using  the uniqueness of weak limits, from \eqref{S3} and \eqref{S8}, we get
		\begin{align}\label{S10}
			\u(\tau;\tau-k,\tilde{\u}_{k})=\tilde{\u}.
		\end{align}
		From \eqref{S*1}, we also have
		\begin{align}\label{S11}
			\frac{\d}{\d t} \|\u\|^2_{\H} +2\alpha\|\u\|^2_{\H} + 2\mu\|\nabla\u\|^2_{\H} +  2\beta \|\u\|^{r+1}_{\wi \L^{r+1}} = 2\left\langle\f,\u\right\rangle.
		\end{align}
		Applying the variation of constant formula to \eqref{S11}, we get that for each $s\in \R$ and $\tau\geq s$,
		\begin{align}\label{S12}
			\|\u(\tau;s,\u_s)\|^2_{\H} &= e^{2\alpha(s-\tau)}\|\u_s\|^2_{\H} -2\mu\int_{s}^{\tau}e^{2\alpha(\xi-\tau)}\|\nabla\u(\xi;s,\u_s)\|^2_{\H}\d\xi\nonumber\\&\quad-2\beta\int_{s}^{\tau}e^{2\alpha(\xi-\tau)}\|\u(\xi;s,\u_s)\|^{r+1}_{\wi \L^{r+1}}\d\xi +2\int_{s}^{\tau}e^{2\alpha(\xi-\tau)}\left\langle\f(\cdot,\xi),\u(\xi;s,\u_s)\right\rangle\d\xi.
		\end{align}
		Using \eqref{S10} in \eqref{S12}, we get 
		\begin{align}\label{S13}
			\|\tilde{\u}\|^2_{\H}&=\|\u(\tau;\tau-k,\tilde{\u}_k)\|^2_{\H} \nonumber\\&= e^{-2\alpha k}\|\tilde{\u}_k\|^2_{\H} -2\mu\int_{\tau-k}^{\tau}e^{2\alpha(\xi-\tau)}\|\nabla\u(\xi;\tau-k,\tilde{\u}_k)\|^2_{\H}\d\xi\nonumber\\&\quad-2\beta\int_{\tau-k}^{\tau}e^{2\alpha(\xi-\tau)}\|\u(\xi;\tau-k,\tilde{\u}_k)\|^{r+1}_{\wi \L^{r+1}}\d\xi\nonumber\\&\quad+2\int_{\tau-k}^{\tau}e^{2\alpha(\xi-\tau)}\left\langle\f(\cdot,\xi),\u(\xi;\tau-k,\tilde{\u}_k)\right\rangle\d\xi.
		\end{align}
		Similarly, \eqref{S6} and \eqref{S12} imply that
		\begin{align}\label{S14}
			\|\u(\tau;\tau-t_n,\u_n^0)\|^2_{\H}&=\|\u(\tau;\tau-k,\u(\tau-k;\tau-t_n,\u_n^0))\|^2_{\H} \nonumber\\&= e^{-2\alpha k}\|\u(\tau-k;\tau-t_n,\u_n^0)\|^2_{\H} \nonumber\\&\quad-2\mu\int_{\tau-k}^{\tau}e^{2\alpha(\xi-\tau)}\|\nabla\u(\xi;\tau-k,\u(\tau-k;\tau-t_n,\u_n^0))\|^2_{\H}\d\xi\nonumber\\&\quad-2\beta\int_{\tau-k}^{\tau}e^{2\alpha(\xi-\tau)}\|\u(\xi;\tau-k,\u(\tau-k;\tau-t_n,\u_n^0))\|^{r+1}_{\wi \L^{r+1}}\d\xi\nonumber\\&\quad+2\int_{\tau-k}^{\tau}e^{2\alpha(\xi-\tau)}\left\langle\f(\cdot,\xi),\u(\xi;\tau-k,\u(\tau-k;\tau-t_n,\u_n^0))\right\rangle\d\xi.
		\end{align}
		Our next aim is to pass the limit as $n\to \infty$ to \eqref{S14}. We estimate the first term by applying the variation of constants formula again to \eqref{S*3} as
		\begin{align}\label{S15}
			&e^{-2\alpha k}\|\u(\tau-k;\tau-t_n,\u_n^0)\|^2_{\H}\leq e^{-\alpha k}\|\u(\tau-k;\tau-t_n,\u_n^0)\|^2_{\H} \nonumber\\&\quad\leq e^{-\alpha t_n}\|\u_n^0\|^2_{\H}+ \frac{e^{-\alpha \tau}}{\min\{\mu,\alpha\}}\int_{-\infty}^{\tau-k}e^{\alpha \xi}\|\f(\cdot,\xi)\|^2_{\V'}\d\xi.
		\end{align}
		Since $\u_n^0\in K(\tau-t_n)$, we get
		\begin{align}\label{S16}
			e^{-\alpha t_n}\|\u_n^0\|^2_{\H}\leq e^{-\alpha t_n}\|K(\tau-t_n)\|^2_{\H}\to 0 \ \text{ as }\  n\to \infty.
		\end{align}
		Making use of the fact given in \eqref{S16}, we obtain from \eqref{S15} that
		\begin{align}\label{S17}
			\limsup_{n\to\infty}e^{-2\alpha k}\|\u(\tau-k,\tau-t_n,\u_n^0)\|^2_{\H} \leq \frac{e^{-\alpha \tau}}{\min\{\mu,\alpha\}}\int_{-\infty}^{\tau-k}e^{\alpha \xi}\|\f(\cdot,\xi)\|^2_{\V'}\d\xi.
		\end{align} 
		By \eqref{S9}, we have
		\begin{align}\label{S18}
			&	\lim_{n\to\infty} 2\int_{\tau-k}^{\tau}e^{2\alpha(\xi-\tau)}\left\langle\f(\cdot,\xi),\u(\xi;\tau-k,\u(\tau-k;\tau-t_n,\u_n^0))\right\rangle\d\xi \nonumber\\&\quad= 2\int_{\tau-k}^{\tau}e^{2\alpha(\xi-\tau)}\left\langle\f(\cdot,\xi),\u(\xi;\tau-k,\tilde{\u}_k)\right\rangle\d\xi.
		\end{align}
		Since $e^{-2\alpha k}\leq e^{2\alpha (\xi-\tau)}\leq 1,$ for $\xi\in(\tau-k,\tau)$, 
		\begin{align*}
			\left(\int_{\tau-k}^{\tau}e^{2\alpha (\xi-\tau)}\|\cdot\|^2_{\H}\d \xi\right)^{1/2} \text{ and } \left(\int_{\tau-k}^{\tau}e^{2\alpha(\xi-\tau)}\|\cdot\|^{r+1}_{\wi \L^{r+1}}\d \xi\right)^{\frac{1}{r+1}} 
		\end{align*}
		define norms, which are equivalent to the standard norms in 
		$$\mathrm{L}^2(\tau-k,\tau;\H) \ \text{ and }\  \mathrm{L}^{r+1}(\tau-k,\tau;\wi\L^{r+1}),$$ 
		respectively. Weak lower semicontinuity property of norm with \eqref{S9} and \eqref{S9'} imply
		\begin{align}\label{S19}
			&\limsup_{n\to\infty}	\left\{-2\mu\int_{\tau-k}^{\tau}e^{2\alpha(\xi-\tau)}\|\nabla\u(\xi;\tau-k,\u(\tau-k;\tau-t_n,\u_n^0))\|^2_{\H}\d\xi\right\} \nonumber\\&\quad\leq - 2\mu\int_{\tau-k}^{\tau}e^{2\alpha(\xi-\tau)}\|\nabla\u(\xi;\tau-k,\tilde{\u}_k)\|^2_{\H}\d\xi,
		\end{align} 
		and 
		\begin{align}\label{S21}
			&	\limsup_{n\to\infty}\left\{-2\beta\int_{\tau-k}^{\tau}e^{2\alpha(\xi-\tau)}\|\u(\xi;\tau-k,\u(\tau-k;\tau-t_n,\u_n^0))\|^{r+1}_{\wi \L^{r+1}}\right\}\nonumber\\&\quad\leq- 2\beta\int_{\tau-k}^{\tau}e^{2\alpha(\xi-\tau)}\|\u(\xi;\tau-k,\tilde{\u}_k)\|^{r+1}_{\wi \L^{r+1}}\d\xi,
		\end{align}
		respectively. Combining \eqref{S17}-\eqref{S21}, and using it in \eqref{S14}, we get
		\begin{align}\label{S22}
			&\limsup_{n\to\infty}\|\u(\tau;\tau-t_n,\u_n^0)\|^2_{\H}\nonumber\\&\leq\frac{e^{-\alpha \tau}}{\min\{\mu,\alpha\}}\int_{-\infty}^{\tau-k} e^{\alpha\xi}  \|\f(\cdot,\xi)\|^2_{\V'}\d \xi-2\mu\int_{\tau-k}^{\tau}e^{2\alpha(\xi-\tau)}\|\nabla\u(\xi;\tau-k,\tilde{\u}_k)\|^2_{\H}\d\xi\nonumber\\&\quad-2\beta\int_{\tau-k}^{\tau}e^{2\alpha(\xi-\tau)}\|\u(\xi;\tau-k,\tilde{\u}_k)\|^{r+1}_{\wi \L^{r+1}}\d\xi+2\int_{\tau-k}^{\tau}e^{2\alpha(\xi-\tau)}\left\langle\f(\cdot,\xi),\u(\xi;\tau-k,\tilde{\u}_k)\right\rangle\d\xi.
		\end{align}
		Now, using \eqref{S13} in \eqref{S22}, we get
		\begin{align}\label{S23}
			\limsup_{n\to\infty}\|\u(\tau;\tau-t_n,\u_n^0)\|^2_{\H}\leq\frac{e^{-\alpha \tau}}{\min\{\mu,\alpha\}}\int_{-\infty}^{\tau-k} e^{\alpha\xi}  \|\f(\cdot,\xi)\|^2_{\V'}\d \xi+\|\tilde{\u}\|^2_{\H}.
		\end{align}
		Passing the limit as $k\to \infty$ in \eqref{S23}, we deduce
		\begin{align}\label{S24}
			\limsup_{n\to\infty}\|\u(\tau;\tau-t_n,\u_n^0)\|^2_{\H}\leq\|\tilde{\u}\|^2_{\H}.
		\end{align}
		By \eqref{S3}-\eqref{S4} and \eqref{S24}, we conclude that
		\begin{align*}
			\lim_{n\to\infty}\|\u(\tau;\tau-t_n,\u_n^0)\|_{\H}=\|\tilde{\u}\|_{\H},
		\end{align*}
		which completes the proof.
	\end{proof}
	We define 
	\begin{align}\label{ab_H1}
		\mathcal{K}_{0}(\tau)=\{\u\in\H:\|\u\|^2_{\H}\leq\mathcal{M}_{0}(\tau)\},
	\end{align}
	where $\mathcal{M}_{0}(\tau)$ is given by
	\begin{align}\label{ab_H2}
		\mathcal{M}_{0}(\tau)=1+ \frac{1}{\min\{\mu,\alpha\}}\int_{-\infty}^{\tau} e^{\alpha(\xi-\tau)} \|\f(\cdot,\xi)\|^2_{\V'}\d \xi.
	\end{align}
	It is immediate from Theorem \ref{D-abH} that $\mathcal{K}_0$ is $\mathfrak{D}_0$-pullback absorbing set of $\Phi_0$ in $\H$. Also, Theorem \ref{D-asymp} implies that $\Phi_0$ is $\mathfrak{D}_0$-pullback asymptotically compact. Hence, by the concept introduced in \cite{CLR} (see Theorem 7 and Remark 14 in \cite{CLR}), we obtain the main result of this section.
	\begin{theorem}\label{Main_T1}
		For $d=2$ with $r\geq1$, $d=3$ with $r>3$ and $d=r=3$ with $2\beta\mu\geq1$, assume that $\f\in\mathrm{L}^2_{\mathrm{loc}}(\R;\V')$ and satisfying \eqref{forcing}. Then, there exists a unique global $\mathfrak{D}_0$-pullback attractor for the continuous cocycle $\Phi_0$ given by \eqref{phi}.
	\end{theorem}
	
	\section{Random dynamical system for non-autonomous stochastic CBF equations}\label{sec4}\setcounter{equation}{0}
	Let us consider the stochastic CBF equations perturbed by multiplicative white noise for $t\geq \tau,$ $\tau\in\mathbb{R}$ as follows:
	\begin{equation}\label{SCBF}
		\left\{
		\begin{aligned}
			\frac{\d\u_{\varepsilon}(t)}{\d t}+\mu \A\u_{\varepsilon}(t)+\B(\u_{\varepsilon}(t))+\alpha\u_{\varepsilon}(t) +\beta\mathcal{C}(\u_{\varepsilon}(t))&=\f(t) +\varepsilon\u_{\varepsilon}(t)\circ\frac{\d \W(t)}{\d t} , \\ 
			\u_{\varepsilon}(x,\tau)&=\u_{\varepsilon,\tau}(x)=\u_{0}(x),
		\end{aligned}
		\right.
	\end{equation}
	$x\in \R^d,$	where  $\varepsilon>0$ and $\circ$ stands for the fact that the stochastic integral is understood in the sense of Stratonovich. The probability space, which will be used later, is denoted by
	\begin{align}\label{Omega}
		\Omega=\{\omega\in\mathrm{C}(\R,\R):\omega(0)=0\}, 
	\end{align}
	where	$\mathscr{F}$ is the Borel sigma-algebra induced by the compact-open topology of $\Omega$, and $\mathbb{P}$ is the two-sided Wiener measure on $(\Omega,\mathscr{F})$. Also, define $\{\vartheta_{t}\}_{t\in\R}$ by 
	\begin{equation}\label{vartheta}
		\vartheta_{t}\omega(\cdot)=\omega(\cdot+t) -\omega(t), \ \ \   t\in\R\ \text{ and }\ \omega\in\Omega.
	\end{equation}
	Hence, $(\Omega,\mathscr{F},\mathbb{P},\{\vartheta_{t}\}_{t\in\R})$ is a metric dynamical system. Moreover, there exists a $\vartheta_{t}$-invariant set $\widetilde{\Omega}\subseteq\Omega$ of full $\mathbb{P}$ measure such that for each $\omega\in\widetilde{\Omega},$  
	\begin{align}\label{omega}
		\frac{\omega(t)}{t}\to 0 \ \text{ as }\  t\to\pm\infty.
	\end{align} 
	Throughout this work, we will not distinguish between $\widetilde{\Omega}$ and $\Omega$. 
	
	Next, for a given $t\in\R$ and $\omega\in \Omega$, let $\z(t,\omega)=e^{-\varepsilon\omega(t)}$. Then, $\z$ satisfies the equation 
	\begin{align}\label{Trans1}
		\d\z=-\varepsilon\z\circ\d \W.
	\end{align}
	Let $\v_{\varepsilon}$ be a new variable given by 
	\begin{align}\label{Trans2}
		\v_{\varepsilon}(t,\tau,\omega,\v_{\varepsilon,\tau})=\z(t,\omega)\u_{\varepsilon}(t,\tau,\omega,\u_{\varepsilon,\tau}) \ \ \
		\text{ with }
		\ \ \	\v_{\varepsilon,\tau}=\z(\tau,\omega)\u_{\varepsilon,\tau}.
	\end{align}
	Then $\v_{\varepsilon}(\cdot)$ satisfies the following:
	\begin{equation}\label{CCBF}
		\left\{
		\begin{aligned}
			\frac{\d\v_{\varepsilon}(t)}{\d t}+\mu \A\v_{\varepsilon}(t)+\frac{1}{\z(t,\omega)}\B\big(\v_{\varepsilon}(t)\big)&+\alpha\v_{\varepsilon}(t)+\frac{\beta}{[\z(t,\omega)]^{r-1}} \mathcal{C}\big(\v_{\varepsilon}(t)\big)\\&=\z(t,\omega)\f(t) , \quad t\geq \tau, \\ 
			\v_{\varepsilon}(x,\tau)&=\v_{\varepsilon,\tau}(x), x\in \R^d, \tau\in\mathbb{R},
		\end{aligned}
		\right.
	\end{equation}
	in $\V'+\widetilde{\L}^{\frac{r+1}{r}}$. By a standard Faedo-Galerkin approximation method, it can be proved that for all $t>\tau, \tau\in\R,$ and for every $\v_{\varepsilon,\tau}\in\H$, \eqref{CCBF} has a unique solution $\v_{\varepsilon}\in\mathrm{C}([\tau,+\infty);\H)\cap\mathrm{L}^2_{\mathrm{loc}}(\tau,+\infty;\V)\cap\mathrm{L}^{r+1}_{\mathrm{loc}}(\tau,+\infty;\widetilde{\L}^{r+1})$ (cf. \cite{HR,PAM}, etc). Moreover, $\v_{\varepsilon}(t,\tau,\omega,\v_{\varepsilon,\tau})$ is continuous with respect to initial data $\v_{\varepsilon,\tau}$ (see Lemma \ref{Continuity} below) and $(\mathscr{F},\mathscr{B}(\H))$-measurable in $\omega\in\Omega.$ This ensure us to define a cocycle $\Phi_{\varepsilon}:\R^+\times\R\times\Omega\times\H\to\H$ for the system \eqref{SCBF} by using \eqref{Trans2}. Given $t\in\R^+, \tau\in\R, \omega\in\Omega$ and $\u_{\varepsilon,\tau}\in\H$, let
	\begin{align}\label{Phi1}
		\Phi_{\varepsilon}(t,\tau,\omega,\u_{\varepsilon,\tau}) =\u_{\varepsilon}(t+\tau,\tau,\vartheta_{-\tau}\omega,\u_{\varepsilon,\tau})=\frac{\v_{\varepsilon}(t+\tau,\tau,\vartheta_{-\tau}\omega,\v_{\varepsilon,\tau})}{\z(t+\tau,\vartheta_{-\tau}\omega)},
	\end{align}
	where $\v_{\varepsilon,\tau}=\z(\tau,\vartheta_{-\tau}\omega)\u_{\varepsilon,\tau}$. By \eqref{Phi1}, for every $t\geq0,\tau\geq0, s\in\R,\omega\in\Omega$ and $\u_{\varepsilon,0}\in\H$, we obtain 
	\begin{align}\label{Phi2}
		\Phi_{\varepsilon}(t+\tau,s,\omega,\u_{\varepsilon,0}) =\frac{\v_{\varepsilon}(t+\tau+s,\tau+s,\vartheta_{-s}\omega,\v_{\varepsilon,0})}{\z(t+\tau+s,\vartheta_{-s}\omega)},
	\end{align}
	where $\v_{\varepsilon,0}=\z(s,\vartheta_{-s}\omega)\u_{\varepsilon,0}$. Similarly, we get
	\begin{align}\label{Phi3}
		&\Phi_{\varepsilon}(t,\tau+s,\vartheta_{\tau}\omega,\Phi_{\varepsilon}(\tau,s,\omega,\u_{\varepsilon,0}))\nonumber\\&\quad =\frac{\v_{\varepsilon}(t+\tau+s,\tau+s,\vartheta_{-s}\omega,\z(\tau+s,\vartheta_{-s}\omega)\Phi_{\varepsilon}(\tau,s,\omega,\u_{\varepsilon,0}))}{\z(t+\tau+s,\vartheta_{-s}\omega)}\nonumber\\&\quad=\frac{\v_{\varepsilon}(t+\tau+s,\tau+s,\vartheta_{-s}\omega,\v_{\varepsilon}(\tau+s,s,\vartheta_{-s}\omega,\v_{\varepsilon,0}))}{\z(t+\tau+s,\vartheta_{-s}\omega)}
		\nonumber\\&\quad=\frac{\v_{\varepsilon}(t+\tau+s,s,\vartheta_{-s}\omega,\v_{\varepsilon,0})}{\z(t+\tau+s,\vartheta_{-s}\omega)}.
	\end{align}
	Therefore, \eqref{Phi2} and \eqref{Phi3} imply that
	\begin{align}\label{Phi4}
		\Phi_{\varepsilon}(t+\tau,s,\omega,\u_{\varepsilon,0})=\Phi_{\varepsilon}(t,\tau+s,\vartheta_{\tau}\omega,\Phi_{\varepsilon}(\tau,s,\omega,\u_{\varepsilon,0})).
	\end{align}
	Since $\v_{\varepsilon}$ is a measurable solution to the system \eqref{CCBF} and continuous with respect to initial data in $\H$, we can see from \eqref{Phi4} that $\Phi_{\varepsilon}$ is continuous cocycle on $\H$ over $(\R,\{\theta_t\}_{t\in\R})$ and $(\Omega,\mathscr{F},\mathbb{P},\{\vartheta_{t}\}_{t\in\R})$, where $\{\theta_t\}_{t\in\R}$ and $\{\vartheta_t\}_{t\in\R}$ are given by \eqref{theta} and \eqref{vartheta}, respectively.
	
	Assume that $\widetilde{D}=\{\widetilde{D}(\tau,\omega):\tau\in\R,\omega\in\Omega\}$ is a family of non-empty subsets of $\H$ satisfying, for every $c>0, \tau\in\R$ and $\omega\in\Omega$, 
	\begin{align}\label{D_1}
		\lim_{t\to\infty}e^{-ct}\|\widetilde{D}(\tau-t,\vartheta_{-t}\omega)\|^2_{\H}=0.
	\end{align}
	Let $\mathfrak{D}$ be the set of all tempered families of bounded non-empty subsets of $\H$, that is,
	\begin{align}\label{D_11}
		\mathfrak{D}=\{\widetilde{D}=\{\widetilde{D}(\tau,\omega):\tau\in\R\text{ and }\omega\in\Omega\}:\widetilde{D} \text{ satisfying } \eqref{D_1}\}.
	\end{align} 
	It is immediate that $\mathfrak{D}$ is neighborhood closed (Definition 2.2, \cite{PeriodicWang}). 
	
	\begin{assumption}
		For proving the results of this section, we need following assumptions on the external forcing term $\f$. There exists a number $\delta\in[0,\alpha)$ such that
		\begin{itemize}
			\item [(i)] 	\begin{align}\label{forcing1}
				\int_{-\infty}^{\tau} e^{\delta\xi}\|\f(\cdot,\xi)\|^2_{\V'}\d \xi<\infty, \ \ \text{ for all } \tau\in\R.
			\end{align}
			\item [(ii)] for every $c>0$
			\begin{align}\label{forcing2}
				\lim_{s\to-\infty}e^{cs}\int_{-\infty}^{0} e^{\delta\xi}\|\f(\cdot,\xi+s)\|^2_{\V'}\d \xi=0.
			\end{align}
		\end{itemize}
		Note that \eqref{forcing2} implies \eqref{forcing1} for $\f\in \mathrm{L}^2_{\mathrm{loc}}(\mathbb{R};\V')$. Also the conditions \eqref{forcing1}-\eqref{forcing2}  do not need $\f$ to be bounded in $\V'$ at $\pm\infty$ (cf. \cite{PeriodicWang}).
	\end{assumption}

	\begin{lemma}\label{Continuity}
		For $d=2$ with $r\geq1$, $d=3$ with $r>3$ and $d=r=3$ with $2\beta\mu\geq1$, assume that $\f\in \mathrm{L}^2_{\mathrm{loc}}(\mathbb{R};\V')$. Then, the solution of \eqref{CCBF} is $(\mathscr{F},\mathscr{B}(\H))$-measurable and continuous in initial data $\v_{\varepsilon,\tau}(x).$
	\end{lemma}
	\begin{proof}
		Let $\v^1_{\varepsilon}(t)$ and $\v^2_{\varepsilon}(t)$ be two solutions of \eqref{CCBF}. Then $\w(\cdot)=\v^1_{\varepsilon}(\cdot)-\v^2_{\varepsilon}(\cdot)$ with $\w(\tau)=\v^1_{\varepsilon,\tau}(x)-\v^2_{\varepsilon,\tau}(x)$ satisfies
		\begin{align}\label{Conti1}
			&	\frac{\d\w(t)}{\d t}+\mu \A\w(t)+\alpha\w(t)\nonumber\\&=-\frac{1}{\z(t,\omega)}\left\{\B\big(\v_{\varepsilon}^1(t)\big)-\B\big(\v_{\varepsilon}^2(t)\big)\right\} -\frac{\beta}{[\z(t,\omega)]^{r-1}}\left\{\mathcal{C}\big(\v_{\varepsilon}^1(t)\big)-\mathcal{C}\big(\v_{\varepsilon}^2(t)\big)\right\},
		\end{align}
		for a.e. $t\in[\tau,\tau+T]$	in $\V'+\widetilde{\L}^{\frac{r+1}{r}}$. Multiplying \eqref{Conti1} with $\w(t)$ and then integrating over $\R^d$, we get
		\begin{align}\label{Conti2}
			&	\frac{1}{2}\frac{\d}{\d t} \|\w(t)\|^2_{\H} +\mu \|\nabla\w(t)\|^2_{\H} + \alpha\|\w(t)\|^2_{\H} \nonumber\\&=-\frac{1}{\z(t,\omega)}\left\langle\B\big(\v_{\varepsilon}^1(t)\big)-\B\big(\v_{\varepsilon}^2(t)\big), \w(t)\right\rangle -\frac{\beta}{[\z(t,\omega)]^{r-1}}\left\langle\mathcal{C}\big(\v_{\varepsilon}^1(t)\big)-\mathcal{C}\big(\v_{\varepsilon}^2(t)\big),\w(t)\right\rangle ,
		\end{align}
		for a.e. $t\in[\tau,\tau+T] \text{ with } T>0$. From \eqref{MO_c}, we have
		\begin{align}\label{Conti3}
			-\beta \frac{1}{[\z(t,\omega)]^{r-1}}\left\langle\mathcal{C}\big(\v_{\varepsilon}^1\big)-\mathcal{C}\big(\v_{\varepsilon}^2\big),\w\right\rangle\leq - \frac{\beta}{2}\frac{1}{[\z(t,\omega)]^{r-1}}\||\w||\v_{\varepsilon}^2|^{\frac{r-1}{2}}\|^2_{\H}\leq0,
		\end{align}
		\vskip 2mm
		\noindent
		\textbf{Case I:} \textit{$d=2$ and $r\geq1$.} Using \eqref{b1}, \eqref{441} and Young's inequality, we obtain
		\begin{align}\label{Conti4}
			\left| \frac{1}{\z(t,\omega)}\left\langle\B\big(\v_{\varepsilon}^1\big)-\B\big(\v_{\varepsilon}^2\big), \w\right\rangle\right|&=\left|\frac{1}{ \z(t,\omega)}\left\langle\B\big(\w,\w \big), \v_{\varepsilon}^2\right\rangle\right|\nonumber\\&\leq\frac{\mu}{2}\|\nabla\w\|^2_{\H}+\frac{C}{[\z(t,\omega)]^4}\|\v_{\varepsilon}^2\|^4_{\widetilde{\L}^4}\|\w\|^2_{\H}\nonumber\\&\leq\frac{\mu}{2}\|\nabla\w\|^2_{\H}+Ce^{4\varepsilon\omega(t)}\|\v_{\varepsilon}^2\|^2_{\H}\|\nabla\v_{\varepsilon}^2\|^2_{\H}\|\w\|^2_{\H}.
		\end{align}
		Making use of \eqref{Conti3} and \eqref{Conti4} in \eqref{Conti2}, we get
		\begin{align*}
			&	\frac{\d}{\d t} \|\w(t)\|^2_{\H}  \leq Ce^{4\varepsilon\omega(t)}\|\v_{\varepsilon}^2\|^2_{\H}\|\nabla\v_{\varepsilon}^2\|^2_{\H}\|\w(t)\|^2_{\H},
		\end{align*} for a.e. $t\in[\tau,\tau+T]$ and an application of Gronwall's inequality implies 
		\begin{align*}
			\|\w(t)\|^2_{\H}&\leq e^{C\int_{\tau}^{t}e^{4\varepsilon\omega(s)}\|\v_{\varepsilon}^2(s)\|^2_{\H}\|\nabla\v_{\varepsilon}^2(s)\|^2_{\H}\d s}\|\w(\tau)\|^2_{\H}\nonumber\\&\leq e^{C\exp\left\{4\varepsilon\sup\limits_{s\in[\tau,\tau+T]}\omega(s)\right\}\sup\limits_{s\in[\tau,\tau+T]}\|\v_{\varepsilon}^2(s)\|^2_{\H}\int\limits_{\tau}^{\tau+T}\|\nabla\v_{\varepsilon}^2(s)\|^2_{\H}\d s}\|\w(\tau)\|^2_{\H}.
		\end{align*} Hence the proof is completed.
		\vskip 2mm
		\noindent
		\textbf{Case II:} \textit{$d= 3$ and $r>3$.} We estimate $|(\B(\w,\w),\v_{\varepsilon}^2)|$ using H\"older's and Young's inequalities as 
		\begin{align}\label{3d-ab12}
			\left|\frac{1}{\z(t,\omega)}(\B(\w,\w),\v_{\varepsilon}^2)\right|&\leq \frac{1}{\z(t,\omega)}\||\v_{\varepsilon}^2||\w|\|_{\H}\|\nabla\w\|_{\H}\nonumber\\&\leq\frac{\mu}{2}\|\nabla\w\|_{\H}^2+\frac{1}{2\mu[\z(t,\omega)]^2}\||\v_{\varepsilon}^2||\w|\|_{\H}^2. 
		\end{align}
		We  estimate the final term from \eqref{3d-ab12} using H\"older's and Young's inequalities as  (similarly as in \cite{MTM1})
		\begin{align}\label{3d-ab13}
			&	\frac{1}{[\z(t,\omega)]^2}\int_{\R^d}|\v_{\varepsilon}^2(x)|^2|\w(x)|^2\d x\nonumber\\&\leq\beta\mu \frac{1}{[\z(t,\omega)]^{r-1}} \left(\int_{\R^d}|\v_{\varepsilon}^2(x)|^{r-1}|\w(x)|^2\d x\right)+\frac{r-3}{r-1}\left[\frac{2}{\beta\mu (r-1)}\right]^{\frac{2}{r-3}}\left(\int_{\R^d}|\w(x)|^2\d x\right).
		\end{align}
		Using the estimate \eqref{3d-ab13} in \eqref{3d-ab12}, we get
		\begin{align}\label{3d-ab14}
			\left|\frac{1}{\z(t,\omega)}(\B(\w,\w),\v_{\varepsilon}^2)\right|&\leq\frac{\mu}{2}\|\nabla\w\|_{\H}^2+\frac{\beta}{2}\frac{1}{[\z(t,\omega)]^{r-1}}\||\w||\v_{\varepsilon}^2|^{\frac{r-1}{2}}\|^2_{\H}+\eta_1\|\w\|^2_{\H},
		\end{align}
		where $\eta_1= \frac{r-3}{2\mu(r-1)}\left[\frac{2}{\beta\mu (r-1)}\right]^{\frac{2}{r-3}}$ . By \eqref{441}, we have 
		\begin{align}\label{Conti7}
			\left|\frac{1}{\z(t,\omega)}\left\langle\B\big(\v_{\varepsilon}^1\big)-\B\big(\v_{\varepsilon}^2\big), \w\right\rangle\right|=\left|\frac{1}{\z(t,\omega)}\left\langle\B\big(\w,\w \big), \v_{\varepsilon}^2\right\rangle\right|.
		\end{align}
		Using  \eqref{Conti3}, \eqref{3d-ab14} and \eqref{Conti7} in \eqref{Conti2}, we get
		\begin{align*}
			&	\frac{\d}{\d t} \|\w(t)\|^2_{\H}  \leq2\eta_1\|\w(t)\|^2_{\H},
		\end{align*} for a.e. $t\in[\tau,\tau+T]$, which implies $\|\w(t)\|^2_{\H}\leq e^{2T\eta_1} \|\w(\tau)\|^2_{\H},$ as required.
		\vskip 2mm
		\noindent
		\textbf{Case III:} \textit{$d=3$ and $r=3$ with $2\beta\mu\geq1$.} Making use of \eqref{b1} and \eqref{441}, we find
		\begin{align}\label{Conti5}
			\left|\frac{1}{\z(t,\omega)}\left\langle\B\big(\v_{\varepsilon}^1\big)-\B\big(\v_{\varepsilon}^2\big), \w\right\rangle\right|&=\left|e^{\varepsilon\omega(t)}\left\langle\B\big(\w,\w \big), \v_{\varepsilon}^2\right\rangle\right|\nonumber\\&\leq\frac{1}{2\beta}\|\nabla\w\|^2_{\H}+\frac{\beta}{2}e^{2\varepsilon\omega(t)}\||\v_{\varepsilon}^2||\w|\|^2_{\H}.
		\end{align}
		Using \eqref{Conti3} and \eqref{Conti5} in \eqref{Conti2}, we get
		\begin{align*}
			&	\frac{1}{2}\frac{\d}{\d t} \|\w(t)\|^2_{\H} +\left(\mu-\frac{1}{2\beta}\right) \|\nabla\w(t)\|^2_{\H} \leq0,
		\end{align*} for a.e. $t\in[\tau,\tau+T]$.
		For $2\beta\mu\geq1$,  we deduce $	\|\w(t)\|^2_{\H}\leq \|\w(\tau)\|^2_{\H},$ which completes the proof.
	\end{proof}
	\section{Random pullback attractor for non-autonomous stochastic CBF equations}\label{sec5}\setcounter{equation}{0}
	This section is devoted to establish the existence of unique random $\mathfrak{D}$-pullback attractor for the system \eqref{SCBF}. We start with deriving uniform estimates on the solution of the system \eqref{CCBF} and then establish the $\mathfrak{D}$-asymptotic compactness of the solution by the concept given in \cite{Ball} for deterministic systems. 
	\begin{lemma}\label{LemmaUe}
		For $d=2$ with $r\geq1$, $d=3$ with $r>3$ and $d=r=3$ with $2\beta\mu\geq1$, assume that $\f\in \mathrm{L}^2_{\mathrm{loc}}(\mathbb{R};\V')$ satisfies \eqref{forcing1}. Then for every $\tau\in\R, \omega\in \Omega$ and $$D=\{D(\tau,\omega):\tau\in\R, \omega\in\Omega\}\in\mathfrak{D},$$ there exists $\mathfrak{T}=\mathfrak{T}(\tau,\omega,D)>0$ such that for all $t\geq \mathfrak{T}$ and $s\geq \tau-t$, the solution $\v_{\varepsilon}(\cdot)$ of the system \eqref{CCBF} with $\omega$ replaced by $\vartheta_{-\tau}\omega$ satisfies 
		\begin{align}
			&	\|\v_{\varepsilon}(s,\tau-t,\vartheta_{-\tau}\omega,\v_{\varepsilon,\tau-t})\|^2_{\H} &\nonumber\\&\quad\leq e^{\alpha(\tau-s)}+ \frac{e^{-\alpha s}}{\min\{\mu,\alpha\}}\int_{-\infty}^{s} e^{\alpha\xi} [\z(\xi,\vartheta_{-\tau}\omega)]^2 \|\f(\cdot,\xi)\|^2_{\V'}\d \xi,\label{ue}
			\\
			&\int_{\tau-t}^{s}e^{\alpha\xi}\|\nabla\v_{\varepsilon}(\xi,\tau-t,\vartheta_{-\tau}\omega,\v_{\varepsilon,\tau-t})\|^2_{\H}\d\xi\nonumber\\&\quad\leq \frac{e^{\alpha\tau}}{\mu}+ \frac{1}{\mu\min\{\mu,\alpha\}}\int_{-\infty}^{s} e^{\alpha\xi} [\z(\xi,\vartheta_{-\tau}\omega)]^2 \|\f(\cdot,\xi)\|^2_{\V'}\d \xi,\label{ue^1}
		\end{align}
		and
		\begin{align}\label{ue^2}
			&\int_{\tau-t}^{s}\frac{e^{\alpha\xi}}{[\z(\xi,\omega)]^{r-1}}\|\v_{\varepsilon}(\xi,\tau-t,\omega,\v_{\varepsilon,\tau-t})\|^{r+1}_{\wi\L^{r+1}}\d\xi\nonumber\\&\leq \frac{e^{\alpha\tau}}{2\beta}+ \frac{1}{2\beta\min\{\mu,\alpha\}}\int_{-\infty}^{s} e^{\alpha\xi} [\z(\xi,\vartheta_{-\tau}\omega)]^2 \|\f(\cdot,\xi)\|^2_{\V'}\d \xi,
		\end{align}
		where $\v_{\varepsilon,\tau-t}\in D(\tau-t,\vartheta_{-t}\omega).$
	\end{lemma}
	\begin{proof}
		From the first equation of the system \eqref{CCBF} and \eqref{b0}, we obtain
		\begin{align}\label{ue0}
			&	\frac{1}{2}\frac{\d}{\d t} \|\v_{\varepsilon}\|^2_{\H} +\mu\|\nabla\v_{\varepsilon}\|^2_{\H} + \alpha\|\v_{\varepsilon}\|^2_{\H} + \frac{\beta}{[\z(t,\omega)]^{r-1}} \|\v_{\varepsilon}(t)\|^{r+1}_{\wi \L^{r+1}}\nonumber\\ &\quad= \z(t,\omega)\left\langle\f,\v_{\varepsilon}\right\rangle\leq \frac{\min\{\mu,\alpha\}}{2}\|\v_{\varepsilon}\|^2_{\V}+\frac{[\z(t,\omega)]^2}{2\min\{\mu,\alpha\}}\|\f\|^2_{\V'},
		\end{align}
		so that 
		\begin{align}\label{ue1}
			\frac{\d}{\d t} \|\v_{\varepsilon}\|^2_{\H}+ \alpha\|\v_{\varepsilon}\|^2_{\H} +\mu\|\nabla\v_{\varepsilon}\|^2_{\H}+ \frac{2\beta}{[\z(t,\omega)]^{r-1}} \|\v_{\varepsilon}\|^{r+1}_{\wi \L^{r+1}} \leq \frac{[\z(t,\omega)]^2}{\min\{\mu,\alpha\}}\|\f\|^2_{\V'}.
		\end{align}
		Applying variation of constant formula to \eqref{ue1}, we obtain
		\begin{align*}
			&	\|\v_{\varepsilon}(s,\tau-t,\omega,\v_{\varepsilon,\tau-t})\|^2_{\H} + \mu\int_{\tau-t}^{s}e^{\alpha(\xi-s)}\|\nabla\v_{\varepsilon}(\xi,\tau-t,\omega,\v_{\varepsilon,\tau-t})\|^2_{\H}\d\xi\nonumber\\& \quad+ 2\beta\int_{\tau-t}^{s}\frac{e^{\alpha(\xi-s)}}{[\z(\xi,\omega)]^{r-1}}\|\v_{\varepsilon}(\xi,\tau-t,\omega,\v_{\varepsilon,\tau-t})\|^{r+1}_{\wi\L^{r+1}}\d\xi\nonumber\\&\leq e^{\alpha(\tau-t-s)}\|\v_{\varepsilon,\tau-t}\|^2_{\H}+ \frac{1}{\min\{\mu,\alpha\}}\int_{\tau-t}^{s} e^{\alpha(\xi-s)} [\z(\xi,\omega)]^2 \|\f(\cdot,\xi)\|^2_{\V'}\d \xi.
		\end{align*}
		Replacing $\omega$ by $\vartheta_{-\tau}\omega$ in the above inequality, we find 
		\begin{align}\label{ue2}
			&	\|\v_{\varepsilon}(s,\tau-t,\vartheta_{-\tau}\omega,\v_{\varepsilon,\tau-t})\|^2_{\H} + \mu\int_{\tau-t}^{s}e^{\alpha(\xi-s)}\|\nabla\v_{\varepsilon}(\xi,\tau-t,\vartheta_{-\tau}\omega,\v_{\varepsilon,\tau-t})\|^2_{\H}\d\xi\nonumber\\&\quad + 2\beta\int_{\tau-t}^{s}\frac{e^{\alpha(\xi-s)}}{[\z(\xi,\vartheta_{-\tau}\omega)]^{r-1}}\|\v_{\varepsilon}(\xi,\tau-t,\vartheta_{-\tau}\omega,\v_{\varepsilon,\tau-t})\|^{r+1}_{\wi\L^{r+1}}\d\xi\nonumber\\&\leq e^{\alpha(\tau-s)}e^{-\alpha t}\|\v_{\varepsilon,\tau-t}\|^2_{\H}+ \frac{e^{-\alpha s}}{\min\{\mu,\alpha\}}\int_{\tau-t}^{s} e^{\alpha\xi} [\z(\xi,\vartheta_{-\tau}\omega)]^2 \|\f(\cdot,\xi)\|^2_{\V'}\d \xi.
		\end{align}
		Since $\v_{\varepsilon,\tau-t}\in D(\tau-t,\vartheta_{-t}\omega)$, we have
		\begin{align}\label{initial_data}
			e^{-\alpha t}\|\v_{\varepsilon,\tau-t}\|^2_{\H}\leq e^{-\alpha t}\|D(\tau-t,\vartheta_{-t}\omega)\|^2_{\H}\to 0, \ \text{ as } \ t \to \infty.
		\end{align} 
		Therefore, there exists $\mathfrak{T}=\mathfrak{T}(\tau,\omega,D)>0$ such that $e^{-\alpha t}\|\v_{\varepsilon,\tau-t}\|^2_{\H}\leq 1$ for all $t\geq \mathfrak{T}.$ Thus 
		\begin{align}\label{ue3}
			e^{\alpha(\tau-s)}e^{-\alpha t}\|\v_{\varepsilon,\tau-t}\|^2_{\H}\leq e^{\alpha(\tau-s)}, \text{ for all } t\geq \mathfrak{T}.
		\end{align}
		Now, it is only left to estimate the final term of \eqref{ue2}. Let $\tilde{\omega}=\vartheta_{-\tau}\omega$. Then by \eqref{omega}, we have that there exists $R<0$ such that for all $\xi\leq R$,
		\begin{align*}
			-2\varepsilon\tilde{\omega}(\xi)\leq-(\alpha-\delta)\xi,
		\end{align*}
		where $\delta$ is the positive constant in \eqref{forcing1}. Therefore, 
		\begin{align*}
			[\z(t,\tilde{\omega})]^2=e^{-2\varepsilon\tilde{\omega}(\xi)}\leq e^{-(\alpha-\delta)\xi}
		\end{align*}
		and we have for all $\xi\leq R$,
		\begin{align*}
			e^{\alpha\xi} [\z(\xi,\vartheta_{-\tau}\omega)]^2 \|\f(\cdot,\xi)\|^2_{\V'}=e^{(\alpha-\delta)\xi} [\z(\xi,\vartheta_{-\tau}\omega)]^2 e^{\delta\xi} \|\f(\cdot,\xi)\|^2_{\V'}\leq e^{\delta\xi} \|\f(\cdot,\xi)\|^2_{\V'}.
		\end{align*}
		Therefore, \eqref{forcing1} gives us that for every $s\in\R, \tau\in\R$ and $\omega\in\Omega$,
		\begin{align}\label{ue4}
			\int_{-\infty}^{s}e^{\alpha\xi} [\z(\xi,\vartheta_{-\tau}\omega)]^2 \|\f(\cdot,\xi)\|^2_{\V'}\d \xi\leq \int_{-\infty}^{s}e^{\delta\xi} \|\f(\cdot,\xi)\|^2_{\V'}\d \xi <\infty.
		\end{align} 
		Hence, from \eqref{ue2}, \eqref{ue3} and \eqref{ue4}, the required results  \eqref{ue}-\eqref{ue^2} follows.
	\end{proof}
	The following Lemma is a direct consequence of Lemma \ref{LemmaUe}.
	\begin{lemma}\label{LemmaUe1}
		For $d=2$ with $r\geq1$, $d=3$ with $r>3$ and $d=r=3$ with $2\beta\mu\geq1$, assume that $\f\in \mathrm{L}^2_{\mathrm{loc}}(\mathbb{R};\V')$ satisfies \eqref{forcing1}. Then for every $\tau\in\R, \omega\in \Omega$ and $$D=\{D(\tau,\omega):\tau\in\R, \omega\in\Omega\}\in\mathfrak{D},$$ there exists $\mathfrak{T}=\mathfrak{T}(\tau,\omega,D)>0$ such that for every $k\geq 0$ and for all $t\geq \mathfrak{T}+k$, the solution $\v_{\varepsilon}(\cdot)$ of the system \eqref{CCBF} with $\omega$ replaced by $\vartheta_{-\tau}\omega$ satisfies 
		\begin{align*}
			&\|\v_{\varepsilon}(\tau-k,\tau-t,\vartheta_{-\tau}\omega,\v_{\varepsilon,\tau-t})\|^2_{\H} \nonumber\\&\quad\leq e^{\alpha k}+ \frac{e^{-\alpha(\tau-k) }}{\min\{\mu,\alpha\}}\int_{-\infty}^{\tau-k} e^{\alpha\xi} [\z(\xi,\vartheta_{-\tau}\omega)]^2 \|\f(\cdot,\xi)\|^2_{\V'}\d \xi,
		\end{align*}
		where $\v_{\varepsilon,\tau-t}\in D(\tau-t,\vartheta_{-t}\omega).$
	\end{lemma} 
	\begin{proof}
		Given $\tau\in \R$ and $k\geq 0$, let $s=\tau-k$. Let $\mathfrak{T}>0$ be the constant claimed in Lemma \ref{LemmaUe}. Also, $t\geq \mathfrak{T}+k$ implies $t\geq \mathfrak{T}$ and $s\geq \tau-t$. Hence, the required result follows from Lemma \ref{LemmaUe}. 
	\end{proof}
	Next, we give an important result on the weak convergence of solution of the system \eqref{CCBF}, which will help us to show the asymptotic compactness of the solution of system \eqref{CCBF}.
	\begin{lemma}\label{weak}
		For $d=2$ with $r\geq1$, $d=3$ with $r>3$ and $d=r=3$ with $2\beta\mu\geq1$, assume that $\f\in\mathrm{L}^2_{\mathrm{loc}}(\R;\V')$. Let $\tau\in\R, \omega\in \Omega$ and $\v_{\varepsilon,\tau}, \v_{\varepsilon,\tau,n}\in \H$ for all $n\in\N.$ If $$\v_{\varepsilon,\tau,n}\xrightharpoonup{w}\v_{\varepsilon,\tau}\ \text{ in }\ \H,$$  then the solution $\v_{\varepsilon}(\cdot)$ of the system \eqref{CCBF} has the following properties:
		\begin{itemize}
			\item [(i)] $\v_{\varepsilon}(\xi,\tau,\omega,\v_{\varepsilon,\tau,n})\xrightharpoonup{w}\v_{\varepsilon}(\xi,\tau,\omega,\v_{\varepsilon,\tau})$ in $\H$  for all  $\xi\geq \tau$.
			\item [(ii)] $\v_{\varepsilon}(\cdot,\tau,\omega,\v_{\varepsilon,\tau,n})\xrightharpoonup{w}\v_{\varepsilon}(\cdot,\tau,\omega,\v_{\varepsilon,\tau})$ in $\mathrm{L}^2(\tau,\tau+T;\V)$ and $\mathrm{L}^{r+1}(\tau,\tau+T;\wi\L^{r+1})$  for every $T>0$.
		\end{itemize}
	\end{lemma}
	\begin{proof}
		This can be proved by the standard method as in \cite{KM} (see Lemmas 5.2 and 5.3 in \cite{KM}).
	\end{proof}
	Next, we prove the pullback asymptotic compactness of the solution of the system \eqref{CCBF}.
	\begin{lemma}\label{Asymptotic_v}
		For $d=2$ with $r\geq1$, $d=3$ with $r>3$ and $d=r=3$ with $2\beta\mu\geq1$, assume that $\f\in\mathrm{L}^2_{\mathrm{loc}}(\R;\V')$ and \eqref{forcing1} holds. Then for every $\tau\in \R, \omega\in \Omega,$ $$D=\{D(\tau,\omega):\tau\in \R,\omega\in \Omega\}\in \mathfrak{D}$$ and $t_n\to \infty, \v_{\varepsilon,0,n}\in D(\tau-t_n, \vartheta_{-t_{n}}\omega)$, the sequence $\v_{\varepsilon}(\tau,\tau-t_n,\vartheta_{-\tau}\omega,\v_{\varepsilon,0,n})$ of solutions of the system \eqref{CCBF} has a convergent subsequence in $\H$.
	\end{lemma}
	\begin{proof}
		From Lemma \ref{LemmaUe1} with $k=0$, it follows  that, there exists $\mathfrak{T}=\mathfrak{T}(\tau,\omega,D)>0$ such that for all $t\geq \mathfrak{T}$,
		\begin{align}\label{ac1}
			\|\v_{\varepsilon}(\tau,\tau-t,\vartheta_{-\tau}\omega,\v_{\varepsilon,\tau-t})\|^2_{\H} \leq 1+ \frac{e^{-\alpha\tau }}{\min\{\mu,\alpha\}}\int_{-\infty}^{\tau} e^{\alpha\xi} [\z(\xi,\vartheta_{-\tau}\omega)]^2 \|\f(\cdot,\xi)\|^2_{\V'}\d \xi,
		\end{align}
		where $\v_{\varepsilon,\tau-t}\in D(\tau-t,\vartheta_{-t}\omega).$ Since $t_n\to \infty$, there exists $N_0\in\N$ such that $t_n\geq \mathfrak{T}$ for all $n\geq N_0$. Since $\v_{\varepsilon,0,n}\in D(\tau-t_n, \vartheta_{-t_{n}}\omega)$, \eqref{ac1} gives that 
		\begin{align}\label{ac2}
			\|\v_{\varepsilon}(\tau,\tau-t_n,\vartheta_{-\tau}\omega,\v_{\varepsilon,0,n})\|^2_{\H} \leq 1+ \frac{e^{-\alpha\tau }}{\min\{\mu,\alpha\}}\int_{-\infty}^{\tau} e^{\alpha\xi} [\z(\xi,\vartheta_{-\tau}\omega)]^2 \|\f(\cdot,\xi)\|^2_{\V'}\d \xi,
		\end{align}
		for all $n\geq N_0.$ From \eqref{ac2}, it is clear that $\v_{\varepsilon}(\tau,\tau-t_n,\vartheta_{-\tau}\omega,\v_{\varepsilon,0,n})$ is bounded in $\H$ for all $n\geq N_0$, and thus there exists $\tilde{\v}_{\varepsilon}\in \H$ and a subsequence (denote by the same notation) such that 
		\begin{align}\label{ac3}
			\v_{\varepsilon}(\tau,\tau-t_n,\vartheta_{-\tau}\omega,\v_{\varepsilon,0,n})\xrightharpoonup{w} \tilde{\v}_{\varepsilon} \text{ in } \H.
		\end{align}
		The above weak convergence also implies that
		\begin{align}\label{ac4}
			\|\tilde{\v}_{\varepsilon}\|_{\H}\leq\liminf_{n\to\infty}\|\v_{\varepsilon}(\tau,\tau-t_n,\vartheta_{-\tau}\omega,\v_{\varepsilon,0,n})\|_{\H}.
		\end{align}
		In order to show that the convergence in \eqref{ac3} is actually strong convergence, we only need to show 
		\begin{align}\label{ac5}
			\|\tilde{\v}_{\varepsilon}\|_{\H}\geq\limsup_{n\to\infty}\|\v_{\varepsilon}(\tau,\tau-t_n,\vartheta_{-\tau}\omega,\v_{\varepsilon,0,n})\|_{\H}.
		\end{align}
		To prove \eqref{ac5}, we use the method of energy equations presented in \cite{Ball}. 	For a given $k\in \N$,	it is easy to see that
		\begin{align}\label{ac6}
			\v_{\varepsilon}(\tau,\tau-t_n,\vartheta_{-\tau}\omega,\v_{\varepsilon,0,n})=\v_{\varepsilon}(\tau,\tau-k,\vartheta_{-\tau}\omega,\v_{\varepsilon}(\tau-k,\tau-t_n,\vartheta_{-\tau}\omega,\v_{\varepsilon,0,n})).
		\end{align} 
		For each $k$, let $N_k$ be sufficiently large such that $t_n\geq \mathfrak{T}+k$ for all $n\geq N_k$. From Lemma \ref{LemmaUe1}, we have 
		\begin{align*}
			\|\v_{\varepsilon}(\tau-k,\tau-t_n,\vartheta_{-\tau}\omega,\v_{\varepsilon,0,n})\|^2_{\H} \leq e^{\alpha k}+ \frac{e^{-\alpha(\tau-k) }}{\min\{\mu,\alpha\}}\int_{-\infty}^{\tau-k} e^{\alpha\xi} [\z(\xi,\vartheta_{-\tau}\omega)]^2 \|\f(\cdot,\xi)\|^2_{\V'}\d \xi,
		\end{align*}
		for $k\geq N_k.$ It is obvious from the above inequality that, for each fixed $k\in \N$, the sequence $\v_{\varepsilon}(\tau-k,\tau-t_n,\vartheta_{-\tau}\omega,\v_{\varepsilon,0,n})$ is bounded in $\H$. By the diagonal process, there exists a subsequence (denoted as the same) and $\tilde{\v}_{\varepsilon,k}\in \H$ for each $k\in\N$ such that 
		\begin{align}\label{ac7}
			\v_{\varepsilon}(\tau-k,\tau-t_n,\vartheta_{-\tau}\omega,\v_{\varepsilon,0,n})\xrightharpoonup{w}\tilde{\v}_{\varepsilon,k} \ \text{ in }\  \H.
		\end{align}
		It follows from \eqref{ac6},\eqref{ac7} and Lemma \ref{weak} that for $k\in\N$,
		\begin{align}\label{ac8}
			\v_{\varepsilon}(\tau,\tau-t_n,\vartheta_{-\tau}\omega,\v_{\varepsilon,0,n})\xrightharpoonup{w} \v_{\varepsilon}(\tau,\tau-k,\vartheta_{-\tau}\omega,\tilde{\v}_{\varepsilon,k}) \ \text{ in } \ \H,
		\end{align}
		\begin{align}\label{ac9}
			\v_{\varepsilon}(\cdot,\tau-k,\vartheta_{-\tau}\omega,\v_{\varepsilon}(\tau-k,&\tau-t_n,\vartheta_{-\tau}\omega,\tilde{\v}_{\varepsilon,k}))\xrightharpoonup{w}\v_{\varepsilon}(\cdot,\tau-k,\vartheta_{-\tau}\omega,\tilde{\v}_{\varepsilon,k})\nonumber\\
			&\text{	in }\mathrm{L}^2(\tau-k,\tau;\V),
		\end{align}
		and
		\begin{align}\label{ac9'}
			\v_{\varepsilon}(\cdot,\tau-k,\vartheta_{-\tau}\omega,\v_{\varepsilon}(\tau-k,&\tau-t_n,\vartheta_{-\tau}\omega,\tilde{\v}_{\varepsilon,k}))\xrightharpoonup{w}\v_{\varepsilon}(\cdot,\tau-k,\vartheta_{-\tau}\omega,\tilde{\v}_{\varepsilon,k})\nonumber\\&\text{	in }\mathrm{L}^{r+1}(\tau-k,\tau;\wi\L^{r+1}).
		\end{align}
		Using \eqref{ac3} and \eqref{ac8}, we get that 
		\begin{align}\label{ac10}
			\v_{\varepsilon}(\tau,\tau-k,\vartheta_{-\tau}\omega,\tilde{\v}_{\varepsilon,k})=\tilde{\v}_{\varepsilon}.
		\end{align}
		From \eqref{ue0}, we also have
		\begin{align}\label{ac11}
			\frac{\d}{\d t} \|\v_{\varepsilon}\|^2_{\H} +2\alpha\|\v_{\varepsilon}\|^2_{\H} + 2\mu\|\nabla\v_{\varepsilon}\|^2_{\H} +  \frac{2\beta}{[\z(t,\omega)]^{r-1}} \|\v_{\varepsilon}\|^{r+1}_{\wi \L^{r+1}} = 2\z(t,\omega)\left\langle\f,\v_{\varepsilon}\right\rangle.
		\end{align}
		Applying the variation of constant formula to \eqref{ac11}, we get that for each $\omega\in \Omega, s\in \R$ and $\tau\geq s$,
		\begin{align}\label{ac12}
			\|\v_{\varepsilon}(\tau,s,\omega,\v_{\varepsilon,s})\|^2_{\H} &= e^{2\alpha(s-\tau)}\|\v_{\varepsilon,s}\|^2_{\H} -2\mu\int_{s}^{\tau}e^{2\alpha(\xi-\tau)}\|\nabla\v_{\varepsilon}(\xi,s,\omega,\v_{\varepsilon,s})\|^2_{\H}\d\xi\nonumber\\&\quad-2\beta\int_{s}^{\tau}\frac{e^{2\alpha(\xi-\tau)}}{[\z(\xi,\omega)]^{r-1}}\|\v_{\varepsilon}(\xi,s,\omega,\v_{\varepsilon,s})\|^{r+1}_{\wi \L^{r+1}}\d\xi\nonumber\\&\quad+2\int_{s}^{\tau}e^{2\alpha(\xi-\tau)}\z(\xi,\omega)\left\langle\f(\cdot,\xi),\v_{\varepsilon}(\xi,s,\omega,\v_{\varepsilon,s})\right\rangle\d\xi.
		\end{align}
		From \eqref{ac10} and \eqref{ac12}, we obtain 
		\begin{align}\label{ac13}
			\|\tilde{\v}_{\varepsilon}\|^2_{\H}&=\|\v_{\varepsilon}(\tau,\tau-k,\vartheta_{-\tau}\omega,\tilde{\v}_{\varepsilon,k})\|^2_{\H} \nonumber\\&= e^{-2\alpha k}\|\tilde{\v}_{\varepsilon,k}\|^2_{\H} -2\mu\int_{\tau-k}^{\tau}e^{2\alpha(\xi-\tau)}\|\nabla\v_{\varepsilon}(\xi,\tau-k,\vartheta_{-\tau}\omega,\tilde{\v}_{\varepsilon,k})\|^2_{\H}\d\xi\nonumber\\&\quad-2\beta\int_{\tau-k}^{\tau}\frac{e^{2\alpha(\xi-\tau)}}{[\z(\xi,\vartheta_{-\tau}\omega)]^{r-1}}\|\v_{\varepsilon}(\xi,\tau-k,\vartheta_{-\tau}\omega,\tilde{\v}_{\varepsilon,k})\|^{r+1}_{\wi \L^{r+1}}\d\xi\nonumber\\&\quad+2\int_{\tau-k}^{\tau}e^{2\alpha(\xi-\tau)}\z(\xi,\vartheta_{-\tau}\omega)\left\langle\f(\cdot,\xi),\v_{\varepsilon}(\xi,\tau-k,\vartheta_{-\tau}\omega,\tilde{\v}_{\varepsilon,k})\right\rangle\d\xi.
		\end{align}
		Similarly, from \eqref{ac6} and \eqref{ac12}, we deduce
		\begin{align}\label{ac14}
			&\|\v_{\varepsilon}(\tau,\tau-t_n,\vartheta_{-\tau}\omega,\v_{\varepsilon,0,n})\|^2_{\H}\nonumber\\&=\|\v_{\varepsilon}(\tau,\tau-k,\vartheta_{-\tau}\omega,\v_{\varepsilon}(\tau-k,\tau-t_n,\vartheta_{-\tau}\omega,\v_{\varepsilon,0,n}))\|^2_{\H} \nonumber\\&= e^{-2\alpha k}\|\v_{\varepsilon}(\tau-k,\tau-t_n,\vartheta_{-\tau}\omega,\v_{\varepsilon,0,n})\|^2_{\H} \nonumber\\&\quad-2\mu\int_{\tau-k}^{\tau}e^{2\alpha(\xi-\tau)}\|\nabla\v_{\varepsilon}(\xi,\tau-k,\vartheta_{-\tau}\omega,\v_{\varepsilon}(\tau-k,\tau-t_n,\vartheta_{-\tau}\omega,\v_{\varepsilon,0,n}))\|^2_{\H}\d\xi\nonumber\\&\quad-2\beta\int_{\tau-k}^{\tau}\frac{e^{2\alpha(\xi-\tau)}}{[\z(\xi,\vartheta_{-\tau}\omega)]^{r-1}}\|\v_{\varepsilon}(\xi,\tau-k,\vartheta_{-\tau}\omega,\v_{\varepsilon}(\tau-k,\tau-t_n,\vartheta_{-\tau}\omega,\v_{\varepsilon,0,n}))\|^{r+1}_{\wi \L^{r+1}}\d\xi\nonumber\\&\quad+2\int_{\tau-k}^{\tau}e^{2\alpha(\xi-\tau)}\z(\xi,\vartheta_{-\tau}\omega)\left\langle\f(\cdot,\xi),\v_{\varepsilon}(\xi,\tau-k,\vartheta_{-\tau}\omega,\v_{\varepsilon}(\tau-k,\tau-t_n,\vartheta_{-\tau}\omega,\v_{\varepsilon,0,n}))\right\rangle\d\xi.
		\end{align}
		Our next aim is to pass to  limit  in each term of the right-hand side of \eqref{ac14} as $n\to \infty$. Using \eqref{ue2}, we estimate the first term with $s=\tau-k$ and $t=t_n$ as follows
		\begin{align}\label{ac15}
			&e^{-2\alpha k}\|\v_{\varepsilon}(\tau-k,\tau-t_n,\vartheta_{-\tau}\omega,\v_{\varepsilon,0,n})\|^2_{\H} \nonumber\\&\quad\leq e^{-\alpha k}\|\v_{\varepsilon}(\tau-k,\tau-t_n,\vartheta_{-\tau}\omega,\v_{\varepsilon,0,n})\|^2_{\H} \nonumber\\&\quad\leq e^{-\alpha t_n}\|\v_{\varepsilon,0,n}\|^2_{\H}+ \frac{e^{-\alpha \tau}}{\min\{\mu,\alpha\}}\int_{-\infty}^{\tau-k} e^{\alpha\xi} [\z(\xi,\vartheta_{-\tau}\omega)]^2 \|\f(\cdot,\xi)\|^2_{\V'}\d \xi.
		\end{align}
		Making use of the temperedness of $\v_{\varepsilon,0,n},$ that is, $\v_{\varepsilon,0,n}\in D(\tau-t_n,\vartheta_{-t_n}\omega)$, we get
		\begin{align}\label{ac16}
			e^{-\alpha t_n}\|\v_{\varepsilon,0,n}\|^2_{\H}\leq e^{-\alpha t_n}\|D(\tau-t_n,\vartheta_{-t_n}\omega)\|^2_{\H}\to 0 \ \text{ as }\  n\to \infty.
		\end{align}
		By  \eqref{ac16}, we obtain from \eqref{ac15} that
		\begin{align}\label{ac17}
			&\limsup_{n\to\infty}e^{-2\alpha k}\|\v_{\varepsilon}(\tau-k,\tau-t_n,\vartheta_{-\tau}\omega,\v_{\varepsilon,0,n})\|^2_{\H}\nonumber\\& \leq \frac{e^{-\alpha \tau}}{\mu}\int_{-\infty}^{\tau-k} e^{\alpha\xi} [\z(\xi,\vartheta_{-\tau}\omega)]^2 \|\f(\cdot,\xi)\|^2_{\V'}\d \xi.
		\end{align} 
		Using the weak convergence given in \eqref{ac9}, we have
		\begin{align}\label{ac18}
			&	\lim_{n\to\infty} 2\int_{\tau-k}^{\tau}e^{2\alpha(\xi-\tau)}\z(\xi,\vartheta_{-\tau}\omega)\left\langle\f(\cdot,\xi),\v_{\varepsilon}(\xi,\tau-k,\vartheta_{-\tau}\omega,\v_{\varepsilon}(\tau-k,\tau-t_n,\vartheta_{-\tau}\omega,\v_{\varepsilon,0,n}))\right\rangle\d\xi \nonumber\\&= 2\int_{\tau-k}^{\tau}e^{2\alpha(\xi-\tau)}\z(\xi,\vartheta_{-\tau}\omega)\left\langle\f(\cdot,\xi),\v_{\varepsilon}(\xi,\tau-k,\vartheta_{-\tau}\omega,\tilde{\v}_{\varepsilon,k})\right\rangle\d\xi.
		\end{align}
		Since $e^{-2\alpha k}\leq e^{2\alpha (\xi-\tau)}\leq 1$ for $\xi\in(\tau-k,\tau)$, 
		\begin{align*}
			\left(\int_{\tau-k}^{\tau}e^{2\alpha (\xi-\tau)}\|\nabla\cdot\|^2_{\H}\d \xi\right)^{1/2} 
		\end{align*}
		defines a norm, which is equivalent to standard norm in $\mathrm{L}^2(\tau-k,\tau;\H)$. Then, the weak lower semicontinuity property of norm implies that
		\begin{align}\label{ac19}
			&\limsup_{n\to\infty}	\left\{-2\mu\int_{\tau-k}^{\tau}e^{2\alpha(\xi-\tau)}\|\nabla\v_{\varepsilon}(\xi,\tau-k,\vartheta_{-\tau}\omega,\v_{\varepsilon}(\tau-k,\tau-t_n,\vartheta_{-\tau}\omega,\v_{\varepsilon,0,n}))\|^2_{\H}\d\xi\right\} \nonumber\\&\leq - 2\mu\int_{\tau-k}^{\tau}e^{2\alpha(\xi-\tau)}\|\nabla\v_{\varepsilon}(\xi,\tau-k,\vartheta_{-\tau}\omega,\tilde{\v}_{\varepsilon,k})\|^2_{\H}\d\xi.
		\end{align} 
		Finally, we estimate the third term of the right-hand side of \eqref{ac14}. Let $\tilde{\omega}=\vartheta_{-\tau}\omega$. By \eqref{omega}, we find that for each $r\geq1$ and for each $\omega\in \Omega$, there exists $R_1<0$ such that for all $\xi\leq R_1$,
		\begin{align}\label{ac20}
			|\varepsilon(r-1)\tilde{\omega}(\xi)|\leq-2\alpha\xi\quad \text{ or } \quad	2\alpha\xi\leq\varepsilon(r-1)\tilde{\omega}(\xi)\leq-2\alpha\xi.
		\end{align}
		By \eqref{ac20}, we have 
		\begin{align*}
			e^{2\alpha(2\xi-\tau)}\leq e^{2\alpha(\xi-\tau)}e^{\varepsilon(r-1)\tilde{\omega}(\xi)}\leq e^{-2\alpha\tau}.
		\end{align*}
		Also, $\frac{e^{2\alpha(\xi-\tau)}}{[\z(\xi,\vartheta_{-\tau}\omega)]^{r-1}}=e^{2\alpha(\xi-\tau)}e^{\varepsilon(r-1)\tilde{\omega}(\xi)}$ and $\xi\in(\tau-k,\tau)$ implies that 
		\begin{align*}
			e^{2\alpha(\tau-2k)}\leq e^{2\alpha(2\xi-\tau)}\leq\frac{e^{2\alpha(\xi-\tau)}}{[\z(\xi,\vartheta_{-\tau}\omega)]^{r-1}}\leq e^{-2\alpha\tau}.
		\end{align*}
		Therefore
		\begin{align*}
			\left(\int_{\tau-k}^{\tau}\frac{e^{2\alpha(\xi-\tau)}}{[\z(\xi,\vartheta_{-\tau}\omega)]^{r-1}}\|\cdot\|^{r+1}_{\wi \L^{r+1}}\d \xi\right)^{\frac{1}{r+1}} 
		\end{align*}
		defines a norm, which is equivalent to the standard norm in $\mathrm{L}^{r+1}(\tau-k,\tau;\wi\L^{r+1})$. Again, using the weak lower semicontinuity property of norm, we obtain that
		\begin{align}\label{ac21}
			&	\limsup_{n\to\infty}\left\{-2\beta\int_{\tau-k}^{\tau}\frac{e^{2\alpha(\xi-\tau)}}{[\z(\xi,\vartheta_{-\tau}\omega)]^{r-1}}\|\v_{\varepsilon}(\xi,\tau-k,\vartheta_{-\tau}\omega,\v_{\varepsilon}(\tau-k,\tau-t_n,\vartheta_{-\tau}\omega,\v_{\varepsilon,0,n}))\|^{r+1}_{\wi \L^{r+1}}\right\}\nonumber\\&\leq- 2\beta\int_{\tau-k}^{\tau}\frac{e^{2\alpha(\xi-\tau)}}{[\z(\xi,\vartheta_{-\tau}\omega)]^{r-1}}\|\v_{\varepsilon}(\xi,\tau-k,\vartheta_{-\tau}\omega,\tilde{\v}_{\varepsilon,k})\|^{r+1}_{\wi \L^{r+1}}.
		\end{align}
		Combining \eqref{ac17}-\eqref{ac19} and \eqref{ac21}, and using it in \eqref{ac14}, we get 
		\begin{align}\label{ac22}
			&\limsup_{n\to\infty}\|\v_{\varepsilon}(\tau,\tau-t_n,\vartheta_{-\tau}\omega,\v_{\varepsilon,0,n})\|^2_{\H}\nonumber\\&\leq\frac{e^{-\alpha \tau}}{\min\{\mu,\alpha\}}\int_{-\infty}^{\tau-k} e^{\alpha\xi} [\z(\xi,\vartheta_{-\tau}\omega)]^2 \|\f(\cdot,\xi)\|^2_{\V'}\d \xi\nonumber\\&\quad-2\mu\int_{\tau-k}^{\tau}e^{2\alpha(\xi-\tau)}\|\nabla\v_{\varepsilon}(\xi,\tau-k,\vartheta_{-\tau}\omega,\tilde{\v}_{\varepsilon,k})\|^2_{\H}\d\xi\nonumber\\&\quad-2\beta\int_{\tau-k}^{\tau}\frac{e^{2\alpha(\xi-\tau)}}{[\z(\xi,\vartheta_{-\tau}\omega)]^{r-1}}\|\v_{\varepsilon}(\xi,\tau-k,\vartheta_{-\tau}\omega,\tilde{\v}_{\varepsilon,k})\|^{r+1}_{\wi \L^{r+1}}\d\xi\nonumber\\&\quad+2\int_{\tau-k}^{\tau}e^{2\alpha(\xi-\tau)}\z(\xi,\vartheta_{-\tau}\omega)\left\langle\f(\cdot,\xi),\v_{\varepsilon}(\xi,\tau-k,\vartheta_{-\tau}\omega,\tilde{\v}_{\varepsilon,k})\right\rangle\d\xi.
		\end{align}
		Now, using \eqref{ac13} in \eqref{ac22}, we deduce
		\begin{align}\label{ac23}
			&\limsup_{n\to\infty}\|\v_{\varepsilon}(\tau,\tau-t_n,\vartheta_{-\tau}\omega,\v_{\varepsilon,0,n})\|^2_{\H}\nonumber\\&\quad\leq\frac{e^{-\alpha \tau}}{\min\{\mu,\alpha\}}\int_{-\infty}^{\tau-k} e^{\alpha\xi} [\z(\xi,\vartheta_{-\tau}\omega)]^2 \|\f(\cdot,\xi)\|^2_{\V'}\d \xi+\|\tilde{\v}_{\varepsilon}\|^2_{\H}.
		\end{align}
		Passing the limit $k\to \infty$ in \eqref{ac23}, we get
		\begin{align}\label{ac24}
			\limsup_{n\to\infty}\|\v_{\varepsilon}(\tau,\tau-t_n,\vartheta_{-\tau}\omega,\v_{\varepsilon,0,n})\|^2_{\H}\leq\|\tilde{\v}_{\varepsilon}\|^2_{\H}.
		\end{align}
		By \eqref{ac3}-\eqref{ac4} and \eqref{ac24}, we conclude the proof.
	\end{proof}
	Now, we are presenting the results for $\u_{\varepsilon}(\cdot)$, that is, the solution of the system \eqref{SCBF}. Since the transformation is same as taken in \cite{PeriodicWang}, \cite{non-autoUpperWang} etc., the proof of following results can be done similarly as in \cite{PeriodicWang} or \cite{non-autoUpperWang}.
	\begin{lemma}\label{LemmaUe3}
		For $d=2$ with $r\geq1$, $d=3$ with $r>3$ and $d=r=3$ with $2\beta\mu\geq1$, assume that $\f\in \mathrm{L}^2_{\mathrm{loc}}(\mathbb{R};\V')$ satisfies \eqref{forcing1}. Then for every $\tau\in\R, \omega\in \Omega$ and $$D=\{D(\tau,\omega):\tau\in\R, \omega\in\Omega\}\in\mathfrak{D},$$ there exists $\mathfrak{T}=\mathfrak{T}(\tau,\omega,D)>0$ such that for all $t\geq \mathfrak{T}$, the solution $\u_{\varepsilon}(\cdot)$ of the system \eqref{SCBF} with $\omega$ replaced by $\vartheta_{-\tau}\omega$ satisfies 
		\begin{align}\label{K0}
			&	\|\u_{\varepsilon}(\tau,\tau-t,\vartheta_{-\tau}\omega,\u_{\varepsilon,\tau-t})\|^2_{\H} \nonumber\\&\leq [\z(\tau,\vartheta_{-\tau}\omega)]^{-2}+ \frac{[\z(\tau,\vartheta_{-\tau}\omega)]^{-2}}{\min\{\mu,\alpha\}}\int_{-\infty}^{\tau} e^{\alpha(\xi-\tau)} [\z(\xi,\vartheta_{-\tau}\omega)]^2 \|\f(\cdot,\xi)\|^2_{\V'}\d \xi,
		\end{align}
		where $\u_{\varepsilon,\tau-t}\in D(\tau-t,\vartheta_{-t}\omega).$
	\end{lemma} 
	\begin{proof}
		The proof can be completed with the help of Lemma \ref{LemmaUe}. See Lemma 5.1 in \cite{PeriodicWang}  also.
	\end{proof}
	\begin{lemma}\label{PA1}
		For $d=2$ with $r\geq1$, $d=3$ with $r>3$ and $d=r=3$ with $2\beta\mu\geq1$, assume that $\f\in \mathrm{L}^2_{\mathrm{loc}}(\mathbb{R};\V')$ satisfies \eqref{forcing2}. Then the continuous cocycle $\Phi_{\varepsilon}$ associated with the system \eqref{SCBF} possesses a closed measurable $\mathfrak{D}$-pullback absorbing set $\mathcal{K}_{\varepsilon}=\{\mathcal{K}_{\varepsilon}(\tau,\omega):\tau\in\R,\omega\in\Omega\}\in\mathfrak{D}$ and $\mathcal{K}_{\varepsilon}(\tau,\omega)$ is denoted by 
		\begin{align}\label{AB1}
			\mathcal{K}_{\varepsilon}(\tau,\omega)=\{\u\in\H:\|\u\|^2_{\H}\leq\mathcal{M}_{\varepsilon}(\tau,\omega)\},
		\end{align}
		where $\mathcal{M}_{\varepsilon}(\tau,\omega)$ is given by
		\begin{align}\label{AB2}
			\mathcal{M}_{\varepsilon}(\tau,\omega)=[\z(\tau,\vartheta_{-\tau}\omega)]^{-2}+ \frac{[\z(\tau,\vartheta_{-\tau}\omega)]^{-2}}{\min\{\mu,\alpha\}}\int_{-\infty}^{\tau} e^{\alpha(\xi-\tau)} [\z(\xi,\vartheta_{-\tau}\omega)]^2 \|\f(\cdot,\xi)\|^2_{\V'}\d \xi.
		\end{align}
	\end{lemma}
	\begin{proof}
		The proof follows by using Lemma \ref{LemmaUe3}. See Lemma 5.2 in \cite{PeriodicWang} also.
	\end{proof}
	\begin{lemma}\label{PA2}
		For $d=2$ with $r\geq1$, $d=3$ with $r>3$ and $d=r=3$ with $2\beta\mu\geq1$, assume that $\f\in \mathrm{L}^2_{\mathrm{loc}}(\mathbb{R};\V')$ satisfies \eqref{forcing2}. Then the continuous cocycle $\Phi_{\varepsilon}$ associated with system \eqref{SCBF} is $\mathfrak{D}$-pullback asymptotically compact in $\H$, that is, for every $\tau\in\R, \omega\in\Omega, D=\{D(\tau,\omega):\tau\in\R,\omega\in\Omega\}\in\mathfrak{D}$, and $t_n\to\infty, \u_{\varepsilon,0,n}\in D(\tau-t_n,\vartheta_{-t_n}\omega),$ the sequence $\Phi_{\varepsilon}(t_n,\tau-t_n,\vartheta_{-t_n}\omega,\u_{\varepsilon,0,n})$ has a convergent subsequence in $\H$. 
	\end{lemma}
	\begin{proof}
		By \eqref{Trans2}, one can write
		\begin{align*}
			\u_{\varepsilon}(\tau,\tau-t_n,\vartheta_{-\tau}\omega,\u_{\varepsilon,0,n})=\frac{\v_{\varepsilon}(\tau,\tau-t_n,\vartheta_{-\tau}\omega,\v_{\varepsilon,0,n})}{\z(\tau,\vartheta_{-\tau}\omega)},
		\end{align*}
		and the proof can be completed using Lemmas \ref{LemmaUe3} and \ref{Asymptotic_v}. See Lemma 5.3  in \cite{PeriodicWang} also.
	\end{proof}
	Now, we present the main result of this section, that is, the existence and uniqueness of tempered pullback attractors for SCBF equations. From Lemmas \ref{PA1} and  \ref{PA2} and the abstract theory of random pullback attractors for non-autonomous non-compact random dynamical system (Theorem 2.23, \cite{SandN_Wang}), we conclude the following result immediately.
	\begin{theorem}\label{PullbackAttractor}
		For $d=2$ with $r\geq1$, $d=3$ with $r>3$ and $d=r=3$ with $2\beta\mu\geq1$, assume that $\f\in \mathrm{L}^2_{\mathrm{loc}}(\mathbb{R};\V')$ satisfies \eqref{forcing2}. Then, the continuous cocycle $\Phi_{\varepsilon}$ associated with system \eqref{SCBF} has a unique random $\mathfrak{D}$-pullback attractor $$\mathscr{A}_{\varepsilon}=\{\mathscr{A}_{\varepsilon}(\tau,\omega):\tau\in\R,\omega\in\Omega\}\in\mathfrak{D}$$ in $\H$. Furthermore, for each $\tau\in\Omega$ and $\omega\in\Omega$,
		\begin{align}
			\mathscr{A}_{\varepsilon}(\tau,\omega)&=\Omega(\mathcal{K}_{\varepsilon},\tau,\omega)=\bigcup_{D\in\mathfrak{D}}\Omega(D,\tau,\omega)\label{AT1}\\
			&=\{\psi(0,\tau,\omega):\psi\text{ is any } \mathfrak{D}\text{-complete orbit of } \Phi_{\varepsilon}\},\label{AT2}
		\end{align}
		where
		\begin{align*}
			\Omega(\mathcal{K}_{\varepsilon},\tau,\omega)=\bigcap_{s\geq0}\overline{\bigcup_{t\geq s}\Phi_{\varepsilon}(t,\tau-t,\vartheta_{-t}\omega,\mathcal{K}_{\varepsilon}(\tau-t,\vartheta_{-t}\omega))}. 
		\end{align*}
	\end{theorem}
	Next, we provide the result on the existence of periodic random pullback attractor for the system \eqref{SCBF}. Assume that $\f:\R\to\V'$ is periodic with period $T>0$ and $\f\in\mathrm{L}^2_{\mathrm{loc}}(\R;\V')$. Then \eqref{forcing2} holds for any $\delta>0$ and  $\Phi_{\varepsilon}$ is also periodic with period $T>0$. Certainly, for every $\tau\in\R, \omega\in\Omega, t\geq0$ and $\bar{\u}\in\H$, we obtain
	\begin{align*}
		\Phi_{\varepsilon}(t,\tau+T,\omega,\bar{\u})&=\u_{\varepsilon}(t+\tau+T,\tau+T,\vartheta_{-\tau-T}\omega,\bar{\u})\nonumber\\&=\u_{\varepsilon}(t+\tau,\tau,\vartheta_{-\tau}\omega,\bar{\u})=\Phi_{\varepsilon}(t,\tau,\omega,\bar{\u}).
	\end{align*}
	Also, by \eqref{K0} we get that the $\mathfrak{D}$-pullback absorbing set $\mathcal{K}_{\varepsilon}$ of $\Phi_{\varepsilon}$ defined by \eqref{K1} is also periodic with period $T>0$, that is, $\mathcal{K}_{\varepsilon}(\tau+T,\omega)=\mathcal{K}_{\varepsilon}(\tau,\omega)$ for all $\tau\in\R$ and $\omega\in\Omega$. Hence, by the theory of periodic random pullback attractor (Theorem 2.24, \cite{SandN_Wang}), we obtain the following result.
	\begin{theorem}\label{PRPA}
		For $d=2$ with $r\geq1$, $d=3$ with $r>3$ and $d=r=3$ with $2\beta\mu\geq1$, assume that $\f:\R\to\V'$ is periodic with period $T>0$ and $\f\in\mathrm{L}^2_{\mathrm{loc}}(\R;\V')$ satisfies \eqref{forcing1}. Then the continuous cocycle $\Phi_{\varepsilon}$ associated with system \eqref{SCBF} has a unique $T$-periodic random $\mathfrak{D}$-pullback attractor $\mathscr{A}_{\varepsilon}\in\mathfrak{D}$ in $\H$.
	\end{theorem}

	\section{Upper semicontinuity of random $\mathfrak{D}$-pullback attractors}\label{sec6}\setcounter{equation}{0}
	In this section, we establish the result on upper semicontinuity of random $\mathfrak{D}$-pullback attractors for the system \eqref{SCBF}, that is, $\mathscr{A}_{\varepsilon}(\tau,\omega)$, when $\varepsilon\to0$, by following the concept introduced in \cite{non-autoUpperWang} for non-autonomous non-compact random dynamical systems. Also, we consider $0<\varepsilon\leq1$ throughout this section. 
	
	Given $0<\varepsilon\leq 1$, It implies from Lemma \ref{PA1} that, for every $D=\{D(\omega)\}_{\omega\in\Omega}\in\mathfrak{D}$ and $\mathbb{P}$-a.e. $\omega\in \Omega$, there exists $\mathfrak{T}(\tau,\omega,D)>0$, which does not depend on $\varepsilon$, such that for all $t\geq \mathfrak{T}(\tau,\omega,D),$
	\begin{align*}
		\|\Phi_{\varepsilon}(t,\tau-t,\vartheta_{-t}\omega,D(\vartheta_{-t}\omega))\|^2_{\H}\leq \mathcal{M}_{\varepsilon}(\tau,\omega),
	\end{align*}
	where $\mathcal{M}_{\varepsilon}(\tau,\omega)$ is given by \eqref{AB2}. Given $0<\varepsilon\leq1$, $\tau\in\R$ and $\omega\in\Omega$, from \eqref{AB1} we write $\mathcal{K}_{\varepsilon}$ for the $\mathfrak{D}$-pullback absorbing set of $\Phi_{\varepsilon}$ as 
	\begin{align}\label{ab_H}
		\mathcal{K}_{\varepsilon}(\tau,\omega)=\{\u\in\H:\|\u\|^2_{\H}\leq \mathcal{M}_{\varepsilon}(\tau,\omega)\}.
	\end{align}
	Given $\tau\in\R$ and $\omega\in\Omega$, denote by 
	\begin{align}\label{ab_H3}
		\mathcal{B}(\tau,\omega)=\{\u\in\H:\|\u\|^2_{\H}\leq \mathcal{R}(\tau,\omega)\}.
	\end{align}
	where $\mathcal{R}(\tau,\omega)$ is given by
	\begin{align}\label{ab_H4}
		\mathcal{R}(\tau,\omega)=e^{2\left|\omega(-\tau)\right|}+ \frac{1}{\min\{\mu,\alpha\}}\int_{-\infty}^{\tau} e^{\alpha(\xi-\tau)} e^{2\left|\omega(\xi-\tau)\right|} \|\f(\cdot,\xi)\|^2_{\V'}\d \xi.
	\end{align}
	Then for every $0<\varepsilon\leq1$, we have
	\begin{align}\label{K1}
		\bigcup_{0<\varepsilon\leq1}\mathcal{K}_{\varepsilon}(\tau,\omega)\subseteq\mathcal{B}(\tau,\omega).
	\end{align}
	Note that Lemma \ref{weak} can also be written in the following form, which will help us to prove Lemma \ref{precompact}.
	\begin{lemma}\label{weak1}
		For $d=2$ with $r\geq1$, $d=3$ with $r>3$ and $d=r=3$ with $2\beta\mu\geq1$, assume that $\f\in\mathrm{L}^2_{\mathrm{loc}}(\R;\V')$. Let $\tau\in\R, \omega\in \Omega, \varepsilon_n\to\varepsilon_0$ and $ \v_{n,j}, \tilde{\v}_{j}\in \H$ for all $n\in\N.$ If $\v_{n,j}\xrightharpoonup{w}\tilde{\v}_{j}$ in $\H$, then the solutions $\v_{\varepsilon_n}(\cdot)$ and $\v_{\varepsilon_0}(\cdot)$ of the system \eqref{CCBF} with $\varepsilon$ replaced by $\varepsilon_n$ and $\varepsilon_0$, respectively. Then, we have
		\begin{itemize}
			\item [(i)] $\v_{\varepsilon_n}(\xi,\tau,\omega,\v_{n,j})\xrightharpoonup{w}\v_{\varepsilon_0}(\xi,\tau,\omega,\tilde{\v}_{j})$ in $\H$ for all $\xi\geq \tau$.
			\item [(ii)] $\v_{\varepsilon_n}(\cdot,\tau,\omega,\v_{n,j})\xrightharpoonup{w}\v_{\varepsilon_0}(\cdot,\tau,\omega,\tilde{\v}_{j})$ in $\mathrm{L}^2(\tau,\tau+T;\V)$ and $\mathrm{L}^{r+1}(\tau,\tau+T;\wi\L^{r+1})$ for every $T>0$.
		\end{itemize}
	\end{lemma}
	
	\begin{lemma}\label{precompact}
		For $0<\varepsilon\leq 1, d=2$ with $r\geq1, d=3$ with $r>3$ and $d=r=3$ with $2\beta\mu\geq1$, assume that $\f\in\mathrm{L}^2_{\mathrm{loc}}(\R;\V')$ and satisfies \eqref{forcing2}. Let $\tau\in\R$ and $\omega\in\Omega$ be fixed. If $\varepsilon_n\to\varepsilon_0$ as $n\to\infty$ and $\u_n\in\mathscr{A}_{\varepsilon_n}(\tau,\omega)$, then the sequence $\{\u_n\}_{n\in\N}$ has a convergent subsequence in $\H$.
	\end{lemma}
	\begin{proof}
		It follows from \eqref{ab_H2} that for every $\tau\in\R$ and $\omega\in\Omega$,
		\begin{align}\label{PC1}
			\lim_{\varepsilon_n\to\varepsilon_0}\mathcal{M}_{\varepsilon_n}(\tau,\omega)=\mathcal{M}_{\varepsilon_0}(\tau,\omega).
		\end{align}
		Since $\varepsilon_n\to 0$, \eqref{PC1} implies that for every $\tau\in\R$ and $\omega\in\Omega$, there exists $\mathcal{N}_1=\mathcal{N}_1(\tau,\omega)$ such that for all $n\geq\mathcal{N}_1$,
		\begin{align}\label{PC2}
			\mathcal{M}_{\varepsilon_n}(\tau,\omega)\leq2\mathcal{M}_{\varepsilon_0}(\tau,\omega).
		\end{align}
		It is given that $\u_n\in\mathscr{A}_{\varepsilon_n}(\tau,\omega)$ and by the property of attractors we know that $\mathscr{A}_{\varepsilon_n}(\tau,\omega)\subseteq\mathcal{K}_{\varepsilon_n}(\tau,\omega),$ by \eqref{AB1} and \eqref{PC2} we obtain, for all $n\geq\mathcal{N}_1$,
		\begin{align}\label{PC3}
			\|\u_n\|^2_{\H}\leq2\mathcal{M}_{\varepsilon_0}(\tau,\omega).
		\end{align}
		It is clear that $\{\u_n\}_{n\in\N}$ is a bounded sequence in $\H$ and therefore, there exists a subsequence (for convenience, we use the same notation) and $\hat{\u}\in\H$ such that 
		\begin{align}\label{PC4}
			\u_n\xrightharpoonup{w}\hat{\u} \ \text{ in }\  \H.
		\end{align}
		In order to complete the proof, we show that the weak convergence in \eqref{PC4} is nothing but a strong convergence in $\H$. Since $\u_n\in\mathscr{A}_{\varepsilon_n}(\tau,\omega)$, by the invariance property of $\mathscr{A}_{\varepsilon_n}(\tau,\omega)$, for all $j\geq1$, there exists $\u_{n,j}\in\mathscr{A}_{\varepsilon_n}(\tau-j,\vartheta_{-j}\omega)$ such that 
		\begin{align}\label{PC5}
			\u_n=\Phi_{\varepsilon_n}(j,\tau-j,\vartheta_{-j}\omega,\u_{n,j})=\u_{\varepsilon_n}(\tau,\tau-j,\vartheta_{-\tau}\omega,\u_{n,j}).
		\end{align}
		Since $\u_{n,j}\in\mathscr{A}_{\varepsilon_n}(\tau-j,\vartheta_{-j}\omega)$ and $\mathscr{A}_{\varepsilon_n}(\tau-j,\vartheta_{-j}\omega)\subseteq\mathcal{K}_{\varepsilon_n}(\tau-j,\vartheta_{-j}\omega),$ by \eqref{AB1} and \eqref{PC2} we get that for each $j\geq1$ and $n\geq\mathcal{N}_1(\tau-j,\vartheta_{-j}\omega)$,
		\begin{align}\label{PC6}
			\|\u_{n,j}\|^2_{\H}\leq2\mathcal{M}_{\varepsilon_0}(\tau-j,\vartheta_{-j}\omega).
		\end{align}
		From \eqref{Trans2}, we infer 
		\begin{align}\label{PC7}
			\v_{\varepsilon_n}(\tau,\tau-j,\vartheta_{-\tau}\omega,\v_{n,j})=e^{-\varepsilon_n\vartheta_{-\tau}\omega(\tau)}\u_{\varepsilon_n}(\tau,\tau-j,\vartheta_{-\tau}\omega,\u_{n,j}),
		\end{align}
		where 
		\begin{align}\label{PC8}
			\v_{n,j}=e^{-\varepsilon_n\vartheta_{-\tau}\omega(\tau-j)}\u_{n,j}.
		\end{align}
		Now, \eqref{PC5} and \eqref{PC7} imply that 
		\begin{align}\label{PC9}
			\u_n=e^{\varepsilon_n\vartheta_{-\tau}\omega(\tau)}\v_{\varepsilon_n}(\tau,\tau-j,\vartheta_{-\tau}\omega,\v_{n,j}).
		\end{align}  
		By \eqref{PC6} and \eqref{PC8} we get, for $n\geq\mathcal{N}_1(\tau-j,\vartheta_{-j}\omega)$,
		\begin{align}\label{PC10}
			\|\v_{n,j}\|^2_{\H}\leq2\mathcal{M}_{\varepsilon_0}(\tau-j,\vartheta_{-j}\omega)e^{-2\varepsilon_n\vartheta_{-\tau}\omega(\tau-j)}.
		\end{align}
		Using \eqref{PC4} and \eqref{PC9} we obtain, as $n\to\infty$,
		\begin{align}\label{PC11}
			\v_{\varepsilon_n}(\tau,\tau-j,\vartheta_{-\tau}\omega,\v_{n,j}) \xrightharpoonup{w}\hat{\v} \text{ in } \H,
		\end{align}
		where $\hat{\v}=e^{-\varepsilon_0\vartheta_{-\tau}\omega(\tau)}\hat{\u}$. By \eqref{PC10}, it is obvious that for each $j\geq1$, the sequence $\{\v_{n,j}\}_{n\in\N}$ is bounded in $\H$, and hence, by a diagonal process, it is easy to find a subsequence (for convenience, keep the labeling same) such that for every $j\geq1$, there exists $\tilde{\v}_j\in\H$ such that
		\begin{align}\label{PC12}
			\v_{n,j}\xrightharpoonup{w}\tilde{\v}_j  \ \text{ in }\ \H \text{ as } n\to\infty.
		\end{align}
		By \eqref{PC12} and Lemma \ref{weak1}, we obtain, as $n\to\infty$,
		\begin{align}\label{PC13}
			\v_{\varepsilon_n}(\tau,\tau-j,\vartheta_{-\tau}\omega,\v_{n,j})\xrightharpoonup{w}\v_{\varepsilon_0}(\tau,\tau-j,\vartheta_{-\tau}\omega,\tilde{\v}_{j})\ \text{ in }\ \H,
		\end{align}
		and
		\begin{align*}
			\v_{\varepsilon_n}(\cdot,\tau-j,\vartheta_{-\tau}\omega,\v_{n,j})\xrightharpoonup{w}\v_{\varepsilon_0}(\cdot,\tau-j,\vartheta_{-\tau}\omega,\tilde{\v}_{j}),
		\end{align*}
		in  $\mathrm{L}^2(\tau-j,\tau;\V)$ and  $\mathrm{L}^{r+1}(\tau-j,\tau;\wi\L^{r+1})$. Now, by \eqref{PC11} and \eqref{PC13}, we get
		\begin{align}\label{PC15}
			\hat{\v}=\v_{\varepsilon_0}(\tau,\tau-j,\vartheta_{-\tau}\omega,\tilde{\v}_{j}).
		\end{align}
		Since we have obtained the same convergence as in Lemma \ref{Asymptotic_v}, further calculation is same as in Lemma \ref{Asymptotic_v}, and hence we omit it here. Following the same calculations as in Lemma \ref{Asymptotic_v} (after \eqref{ac10}), that is, using the idea of energy equations, we obtain
		\begin{align}
			\v_{\varepsilon_n}(\tau,\tau-j,\vartheta_{-\tau}\omega,\v_{n,j}) \to\hat{\v}\  \text{ in } \ \H,
		\end{align}
		where $\hat{\v}=e^{-\varepsilon_0\vartheta_{-\tau}\omega(\tau)}\hat{\u}$ or equivalently
		\begin{align}\label{PC16}
			\u_{\varepsilon_n}(\tau,\tau-j,\vartheta_{-\tau}\omega,\u_{n,j}) \to\hat{\u} \ \text{ in } \ \H.
		\end{align}
		From \eqref{PC5} and \eqref{PC16}, we infer that
		\begin{align}
			\u_{n} \to\hat{\u} \ \text{ in }\  \H,
		\end{align}
		which completes the proof.
	\end{proof}
	\begin{theorem}\label{Perturbation_v}
		For $0<\varepsilon\leq 1, d=2$ with $r\geq1, d=3$ with $r>3$ and $d=r=3$ with $2\beta\mu\geq1$, assume that $\f\in\mathrm{L}^2_{\mathrm{loc}}(\R;\V')$.  Let $\v_{\varepsilon}$ and $\u$ be the solutions of systems \eqref{CCBF} and \eqref{CBF}, respectively. Then, for every $\tau\in\R, \omega\in \Omega, T>0$ and $t\in[\tau,\tau+T]$,
		\begin{align}\label{P}
			\lim_{\varepsilon\to0}\|\v_{\varepsilon}(t,\tau,\omega,\v_{\varepsilon,\tau})-\u(t;\tau,\u_{\tau})\|^2_{\H}=0.
		\end{align}
	\end{theorem}
	\begin{proof}
		Let $\y_{\varepsilon}=\v_{\varepsilon}-\u$. Then from \eqref{CCBF} and \eqref{CBF}, we get
		\begin{align}\label{P1}
			\frac{\d\y_{\varepsilon}(t)}{\d t}+\mu \A\y_{\varepsilon}(t)+\alpha\y_{\varepsilon}(t)&=-\frac{1}{\z(t,\omega)}\B\big(\v_{\varepsilon}(t)\big)+\B\big(\u(t)\big)-\frac{\beta}{[\z(t,\omega)]^{r-1}} \mathcal{C}\big(\v_{\varepsilon}(t)\big)\nonumber\\&\quad+\beta\mathcal{C}\big(\u(t)\big) +(\z(t,\omega)-1)\f(t).
		\end{align}
		Taking the inner product with $\y_{\varepsilon}$ to the equation \eqref{P1}, we obtain
		\begin{align}\label{P3}
			\frac{1}{2}\frac{\d}{\d t}\|\y_{\varepsilon}\|^2_{\H}&=-\mu\|\nabla\y_{\varepsilon}\|^2_{\H}-\alpha\|\y_{\varepsilon}\|^2_{\H}-e^{\varepsilon\omega(t)}b(\v_{\varepsilon},\v_{\varepsilon},\y_{\varepsilon}) +b(\u,\u,\y_{\varepsilon}) \nonumber\\&\quad-e^{\varepsilon(r-1)\omega(t)}\left\langle\mathcal{C}(\v_{\varepsilon}),\y_{\varepsilon}\right\rangle+\left\langle\mathcal{C}(\u),\y_{\varepsilon}\right\rangle+(e^{-\varepsilon\omega(t)}-1)\left\langle\f,\y_{\varepsilon}\right\rangle\nonumber\\&=-\mu\|\nabla\y_{\varepsilon}\|^2_{\H}-\alpha\|\y_{\varepsilon}\|^2_{\H}-e^{\varepsilon\omega(t)}b(\y_{\varepsilon},\u,\y_{\varepsilon}) -(e^{\varepsilon\omega(t)}-1)b(\u,\u,\y_{\varepsilon}) \nonumber\\&\quad-e^{\varepsilon(r-1)\omega(t)}\left\langle\mathcal{C}(\v_{\varepsilon})-\mathcal{C}(\u),\v_{\varepsilon}-\u\right\rangle-(e^{\varepsilon(r-1)\omega(t)}-1)\left\langle\mathcal{C}(\u),\y_{\varepsilon}\right\rangle\nonumber\\&\quad+(e^{-\varepsilon\omega(t)}-1)\left\langle\f,\y_{\varepsilon}\right\rangle.
		\end{align}
		By \eqref{MO_c}, we have 
		\begin{align}\label{P8}
			&-e^{\varepsilon(r-1)\omega(t)}\left\langle\mathcal{C}(\v_{\varepsilon})-\mathcal{C}(\u),\v_{\varepsilon}-\u\right\rangle\nonumber\\ &\leq-\frac{\beta}{2}e^{\varepsilon(r-1)\omega(t)}\||\y_{\varepsilon}||\v_{\varepsilon}|^{\frac{r-1}{2}}\|^2_{\H}-\frac{\beta}{2}e^{\varepsilon(r-1)\omega(t)}\||\y_{\varepsilon}||\u|^{\frac{r-1}{2}}\|^2_{\H}\leq0.
		\end{align}
		\vskip 2mm
		\noindent
		\textbf{Case I:} \textit{$d=2$ and $r\geq1$.} Applying \eqref{b1}, H\"older's and Young's inequalities, we estimate the right hand side terms of \eqref{P3} as follows
		\begin{align}
			\left|e^{\varepsilon\omega(t)}b(\y_{\varepsilon},\u,\y_{\varepsilon})\right|&\leq Ce^{\varepsilon\omega(t)}\|\y_{\varepsilon}\|_{\H}\|\nabla\y_{\varepsilon}\|_{\H}\|\nabla\u\|_{\H}\nonumber\\&\leq Ce^{2\varepsilon\omega(t)}\|\nabla\u\|^2_{\H}\|\y_{\varepsilon}\|^2_{\H}+\frac{\mu}{4}\|\nabla\y_{\varepsilon}\|^2_{\H},\label{P4}\\
			\left|(e^{\varepsilon\omega(t)}-1)b(\u,\u,\y_{\varepsilon})\right|&\leq C\left|e^{\varepsilon\omega(t)}-1\right|\|\u\|^{1/2}_{\H}\|\nabla\u\|^{3/2}_{\H}\|\y_{\varepsilon}\|^{1/2}_{\H}\|\nabla\y_{\varepsilon}\|^{1/2}_{\H}\nonumber\\&\leq C\|\y_{\varepsilon}\|^2_{\H}\|\nabla\y_{\varepsilon}\|^2_{\H} + C\left|e^{\varepsilon\omega(t)}-1\right|^{4/3} \|\u\|^{2/3}_{\H}\|\nabla\u\|^2_{\H}\nonumber\\&\leq C\|\nabla\v_{\varepsilon}\|^2_{\H}\|\y_{\varepsilon}\|^2_{\H}+C\|\nabla\u\|^2_{\H}\|\y_{\varepsilon}\|^2_{\H} \nonumber\\&\quad+ C|e^{\varepsilon\omega(t)}-1|^{4/3} \|\u\|^{2/3}_{\H}\|\nabla\u\|^2_{\H},\label{P5}\\
			\left|(e^{\varepsilon(r-1)\omega(t)}-1)\left\langle\mathcal{C}(\u),\y_{\varepsilon}\right\rangle\right|&\leq\left|e^{\varepsilon(r-1)\omega(t)}-1\right|\|\u\|^r_{\widetilde{\L}^{r+1}}\|\y_{\varepsilon}\|_{\widetilde{\L}^{r+1}}\nonumber\\&\leq C\left|e^{\varepsilon(r-1)\omega(t)}-1\right|\|\u\|^{r+1}_{\widetilde{\L}^{r+1}}+C\left|e^{\varepsilon(r-1)\omega(t)}-1\right|\|\y_{\varepsilon}\|^{r+1}_{\widetilde{\L}^{r+1}},\label{P6}\\
			(e^{-\varepsilon\omega(t)}-1)\left\langle\f,\y_{\varepsilon}\right\rangle&\leq\frac{1}{\min\{\mu,\alpha\}}\left|e^{-\varepsilon\omega(t)}-1\right|^2\|\f\|^2_{\V'}+\frac{\min\{\mu,\alpha\}}{4}\|\y_{\varepsilon}\|^2_{\V}.\label{P7}
		\end{align}
		Combining \eqref{P8}-\eqref{P7} and putting in \eqref{P3}, we obtain
		\begin{align}\label{P9}
			\frac{\d}{\d t}\|\y_{\varepsilon}(t)\|^2_{\H}&\leq	P_1(t)\|\y_{\varepsilon}(t)\|^2_{\H} + P_2(t),
		\end{align}
		for a.e. $t\in[\tau,\tau+T]$, where
		\begin{align*}
			P_1(t)&=	C\left[(e^{2\varepsilon\omega(t)}+1)\|\nabla\u(t)\|^2_{\H}+\|\nabla\v_{\varepsilon}(t)\|^2_{\H}\right],\\
			P_2(t)&=C|e^{\varepsilon(r-1)\omega(t)}-1|\|\u(t)\|^{r+1}_{\widetilde{\L}^{r+1}}+C|e^{\varepsilon(r-1)\omega(t)}-1|\|\v_{\varepsilon}(t)\|^{r+1}_{\widetilde{\L}^{r+1}}\nonumber\\&\quad+C|e^{\varepsilon\omega(t)}-1|^{4/3} \|\u(t)\|^{2/3}_{\H}\|\nabla\u(t)\|^2_{\H}+\frac{2|e^{-\varepsilon\omega(t)}-1|^2}{\min\{\mu,\alpha\}}\|\f(t)\|^2_{\V'}.
		\end{align*}
		Now, applying Gronwall's inequality, we obtain that, for all $t\in[\tau,\tau+T]$
		\begin{align}\label{P10}
			&\|\v_{\varepsilon}(t,\tau,\omega,\v_{\varepsilon,\tau})-\u(t;\tau,\u_{\tau})\|^2_{\H}\nonumber\\&\leq\bigg(\|e^{-\varepsilon\omega(t)}\u_0-\u_0\|^2_{\H}+C\int_{\tau}^{\tau+T}\big\{|e^{\varepsilon(r-1)\omega(s)}-1|\|\u(s)\|^{r+1}_{\widetilde{\L}^{r+1}}+|e^{\varepsilon(r-1)\omega(s)}-1|\|\v_{\varepsilon}(s)\|^{r+1}_{\widetilde{\L}^{r+1}}\nonumber\\&\qquad+|e^{\varepsilon\omega(s)}-1|^{4/3}\|\u(s)\|^{2/3}_{\H}\|\nabla\u(s)\|^2_{\H}+\frac{2|e^{-\varepsilon\omega(s)}-1|^2}{\min\{\mu,\alpha\}}\|\f(s)\|^2_{\V'}\big\}\d s\bigg)\nonumber\\&\qquad\times e^{C\int_{\tau}^{\tau+T}\left\{(e^{2\varepsilon\omega(s)}+1)\|\nabla\u(s)\|^2_{\H}+\|\nabla\v_{\varepsilon}(s)\|^2_{\H}\right\}\d s}\nonumber\\&\leq\bigg(\|e^{-\varepsilon\omega(t)}\u_0-\u_0\|^2_{\H}+C\sup_{s\in[\tau,\tau+T]}|e^{\varepsilon(r-1)\omega(s)}-1|\int_{\tau}^{\tau+T}\left[\|\u(s)\|^{r+1}_{\widetilde{\L}^{r+1}}+\|\v_{\varepsilon}(s)\|^{r+1}_{\widetilde{\L}^{r+1}}\right]\d s\nonumber\\&\qquad+C\sup_{s\in[\tau,\tau+T]}|e^{\varepsilon\omega(s)}-1|^{4/3}\sup_{s\in[\tau,\tau+T]}\|\u(s)\|^{2/3}_{\H}\int_{\tau}^{\tau+T}\|\nabla\u(s)\|^2_{\H}\d s\nonumber\\&\qquad+\frac{2\sup\limits_{s\in[\tau,\tau+T]}|e^{-\varepsilon\omega(s)}-1|^2}{\min\{\mu,\alpha\}}\int_{\tau}^{\tau+T}\|\f(s)\|^2_{\V'}\d s\bigg)\nonumber\\&\qquad\times e^{C\sup\limits_{s\in[\tau,\tau+T]}\left|e^{2\varepsilon\omega(s)}+1\right|\int_{\tau}^{\tau+T}\|\nabla\u(s)\|^2_{\H}\d s+C\int_{\tau}^{\tau+T}\|\nabla\v_{\varepsilon}(s)\|^2_{\H}\d s},
		\end{align}
		where we have used the fact $\omega(t)$ is continuous on $\R$. Since $$\u,\v_{\varepsilon}\in\mathrm{C}([\tau,\tau+T];\H)\cap\mathrm{L}^2(\tau,\tau+T;\V)\cap\mathrm{L}^{r+1}(\tau,\tau+T;\widetilde{\L}^{r+1}) \text{ and } \f\in\mathrm{L}^2_{\mathrm{loc}}(\R;\V'),$$ we can take limit $\varepsilon\to0$ in \eqref{P10}, which completes the proof.
		\vskip 2mm
		\noindent
		\textbf{Case II:} \textit{$d=3$ and $r>3$.} Similar calculation as is \eqref{3d-ab14} and \eqref{3d-ab13}  imply
		\begin{align}\label{P11}
			\left|e^{\varepsilon\omega(t)}b(\y_{\varepsilon},\u,\y_{\varepsilon})\right|&=\left|e^{\varepsilon\omega(t)}b(\y_{\varepsilon},\y_{\varepsilon},\u)\right|\nonumber\\&\leq\frac{\mu}{4}\|\nabla\y_{\varepsilon}\|_{\H}^2+\frac{\beta}{4}e^{\varepsilon(r-1)\omega(t)}\||\y_{\varepsilon}||\u|^{\frac{r-1}{2}}\|^2_{\H}+C\|\y_{\varepsilon}\|^2_{\H}
		\end{align}
		and 
		\begin{align}\label{P12}
			&\left|(e^{\varepsilon\omega(t)}-1)b(\u,\u,\y_{\varepsilon})\right|\nonumber\\&\quad\leq\left|1-e^{-\varepsilon\omega(t)}\right|e^{\varepsilon\omega(t)}\|\nabla\u\|_{\H}\||\y_{\varepsilon}||\u|\|_{\H}\nonumber\\&\quad\leq\frac{1}{2\beta}\left|1-e^{-\varepsilon\omega(t)}\right|^2\|\nabla\u\|^2_{\H}+\frac{\beta}{2}e^{2\varepsilon\omega(t)}\||\y_{\varepsilon}||\u|\|^2_{\H}\nonumber\\&\quad\leq\frac{1}{2\beta}\left|1-e^{-\varepsilon\omega(t)}\right|^2\|\nabla\u\|^2_{\H}+\frac{\beta}{4}e^{\varepsilon(r-1)\omega(t)}\||\y_{\varepsilon}||\u|^{\frac{r-1}{2}}\|^2_{\H}+C\|\y_{\varepsilon}\|^2_{\H}.
		\end{align}
		Combining \eqref{P3}-\eqref{P8}, \eqref{P6}-\eqref{P7} and \eqref{P11}-\eqref{P12}, we obtain
		\begin{align}\label{P14}
			\frac{\d}{\d t}\|\y_{\varepsilon}(t)\|^2_{\H}&\leq	C\|\y_{\varepsilon}(t)\|^2_{\H} + P(t),
		\end{align}
		for a.e. $t\in[\tau,\tau+T]$, where
		\begin{align*}
			P(t)&=C|e^{\varepsilon(r-1)\omega(t)}-1|\|\u(t)\|^{r+1}_{\widetilde{\L}^{r+1}}+C|e^{\varepsilon(r-1)\omega(t)}-1|\|\v_{\varepsilon}(t)\|^{r+1}_{\widetilde{\L}^{r+1}}\nonumber\\&\quad+\frac{1}{\beta}\left|1-e^{-\varepsilon\omega(t)}\right|^2\|\nabla\u(t)\|^2_{\H}+\frac{2|e^{-\varepsilon\omega(t)}-1|^2}{\min\{\mu,\alpha\}}\|\f(t)\|^2_{\V'}.
		\end{align*}
		Now, applying Gronwall's inequality, we get that, for all $t\in[\tau,\tau+T]$
		\begin{align}\label{P15}
			&\|\v_{\varepsilon}(t,\tau,\omega,\v_{\varepsilon,\tau})-\u(t;\tau,\u_{\tau})\|^2_{\H}\nonumber\\&\leq\bigg(\|e^{-\varepsilon\omega(t)}\u_0-\u_0\|^2_{\H}+C\sup_{s\in[\tau,\tau+T]}|e^{\varepsilon(r-1)\omega(s)}-1|\int_{\tau}^{\tau+T}\left[\|\u(s)\|^{r+1}_{\widetilde{\L}^{r+1}}+\|\v_{\varepsilon}(s)\|^{r+1}_{\widetilde{\L}^{r+1}}\right]\d s\nonumber\\&\qquad+\frac{1}{\beta}\sup_{s\in[\tau,\tau+T]}|1-e^{-\varepsilon\omega(s)}|^{2}\int_{\tau}^{\tau+T}\|\nabla\u(s)\|^2_{\H}\d s+\frac{2\sup\limits_{s\in[\tau,\tau+T]}|e^{-\varepsilon\omega(s)}-1|^2}{\min\{\mu,\alpha\}}\nonumber\\&\qquad\times\int_{\tau}^{\tau+T}\|\f(s)\|^2_{\V'}\d s\bigg) e^{CT},
		\end{align}
		where we have used the fact $\omega(t)$ is continuous on $\R$. Since $$\u,\v_{\varepsilon}\in\mathrm{L}^2(\tau,\tau+T;\V)\cap\mathrm{L}^{r+1}(\tau,\tau+T;\widetilde{\L}^{r+1}) \text{ and } \f\in\mathrm{L}^2_{\mathrm{loc}}(\R;\V'),$$ we can take limit $\varepsilon\to0$ in \eqref{P10}, which completes the proof.
		\vskip 2mm
		\noindent
		\textbf{Case III:} \textit{When $d=r=3$ with $2\beta\mu\geq1$.} Applying H\"older's and Young's inequalities, we obtain
		\begin{align}
			\left|e^{\varepsilon\omega(t)}b(\y_{\varepsilon},\u,\y_{\varepsilon})\right|&=\left|e^{\varepsilon\omega(t)}b(\y_{\varepsilon},\y_{\varepsilon},\v_{\varepsilon})\right|\leq\frac{1}{2\beta}\|\nabla\y_{\varepsilon}\|_{\H}^2+\frac{\beta}{2}e^{2\varepsilon\omega(t)}\||\y_{\varepsilon}||\v_{\varepsilon}|\|^2_{\H}\label{P16}\\
			\left|(e^{\varepsilon\omega(t)}-1)b(\u,\u,\y_{\varepsilon})\right|&\leq\left|1-e^{-\varepsilon\omega(t)}\right|e^{\varepsilon\omega(t)}\|\nabla\u\|_{\H}\||\y_{\varepsilon}||\u|\|_{\H}\nonumber\\&\leq\frac{1}{2\beta}\left|1-e^{-\varepsilon\omega(t)}\right|^2\|\nabla\u\|^2_{\H}+\frac{\beta}{2}e^{2\varepsilon\omega(t)}\|\|\y_{\varepsilon}||\u|\|^2_{\H}.\label{P17}
		\end{align}
		Combining \eqref{P3}-\eqref{P8}, \eqref{P6}-\eqref{P7} and \eqref{P16}-\eqref{P17}, we obtain
		\begin{align}\label{P19}
			\frac{\d}{\d t}\|\y_{\varepsilon}(t)\|^2_{\H}&\leq	C\|\y_{\varepsilon}(t)\|^2_{\H} + P(t),
		\end{align}
		for a.e. $t\in[\tau,\tau+T]$, where $P(t)$ is same as defined in \eqref{P14} with $r=3$. Since $\u,\v_{\varepsilon}\in\mathrm{L}^2(\tau,\tau+T;\V)\cap\mathrm{L}^{4}(\tau,\tau+T;\widetilde{\L}^{4}) \text{ and } \f\in\mathrm{L}^2_{\mathrm{loc}}(\R;\V'),$ arguing similarly as in the for case $d=3$ with $r>3$, one can complete the proof.
	\end{proof}
	The following result is an immediate consequence of  Theorem \ref{Perturbation_v}.
	\begin{corollary}\label{Perturbation_u}
		For $0<\varepsilon\leq 1, d=2$ with $r\geq1, d=3$ with $r>3$ and $d=r=3$ with $2\beta\mu\geq1$, assume that $\f\in\mathrm{L}^2_{\mathrm{loc}}(\R;\H)$.  Let $\u_{\varepsilon}$ and $\u$ be the solutions of systems \eqref{SCBF} and \eqref{CBF}, respectively. Then, for every $\tau\in\R, \omega\in \Omega, T>0$ and $t\in[\tau,\tau+T]$,
		\begin{align}\label{P*}
			\lim_{\varepsilon\to0}\|\u_{\varepsilon}(t,\tau,\omega,\u_{\varepsilon,\tau})-\u(t;\tau,\u_{\tau})\|_{\H}=0.
		\end{align}
	\end{corollary}
	\begin{proof}
		The estimate 
		\begin{align*}
			&\lim_{\varepsilon\to0}\|\u_{\varepsilon}(t,\tau,\omega,\u_{\varepsilon,\tau})-\u(t;\tau,\u_{\tau})\|_{\H}\nonumber\\&\quad\leq\lim_{\varepsilon\to0}\left|e^{\varepsilon\omega(t)}-1\right|\|\v_{\varepsilon}(t,\tau,\omega,\u_{\varepsilon,\tau})\|_{\H}+\lim_{\varepsilon\to0}\|\v_{\varepsilon}(t,\tau,\omega,\u_{\varepsilon,\tau})-\u(t;\tau,\u_{\tau})\|^2_{\H},
		\end{align*}
		concludes the proof using Theorem \ref{Perturbation_v}.
	\end{proof}
	We are now ready to prove the main result of this section, that is, upper semicontinuity of random pullback attractors for the system \eqref{SCBF}. The concept of upper semicontinuity of non-compact non-autonomous random dynamical system was introduced in \cite{non-autoUpperWang} (see Theorem 3.2, \cite{non-autoUpperWang}). 
	\begin{theorem}\label{Main_T3}
		For $0<\varepsilon\leq 1, d=2$ with $r\geq1, d=3$ with $r>3$ and $d=r=3$ with $2\beta\mu\geq1$, assume that $\f\in\mathrm{L}^2_{\mathrm{loc}}(\R;\V')$ and \eqref{forcing2} hold. Then for every $\omega\in \Omega$ and $\tau\in\R$,
		\begin{align}\label{U-SC}
			\lim_{\varepsilon\to0}\emph{dist}_{\H}\left(\mathscr{A}_{\varepsilon}(\tau,\omega),\mathscr{A}(\tau)\right)=0.
		\end{align}
	\end{theorem}
	\begin{proof}
		Let $\mathcal{K}_{\varepsilon}(\tau,\omega)$ and $\mathcal{K}_{0}(\tau)$ be the families of subsets of $\H$ given by \eqref{ab_H} and \eqref{ab_H1}, respectively. Also, let $\mathcal{K}_{\varepsilon}$ be a $\mathfrak{D}$-pullback absorbing set of $\Phi_{\varepsilon}$ and $\mathcal{K}_{0}$ be a $\mathfrak{D}_0$-pullback absorbing set of $\Phi_0$ in $\H$. By \eqref{AB1} and \eqref{ab_H2}, we obtain for every  $\tau\in\R$ and $\omega\in\Omega$,
		\begin{align}\label{U-SC1}
			\limsup_{\varepsilon\to0}\|\mathcal{K}_{\varepsilon}(\tau,\omega)\|^2_{\H}\leq\limsup_{\varepsilon\to0}\mathcal{M}_{\varepsilon}(\tau,\omega)=\mathcal{M}_0(\tau).
		\end{align}
		Consider a sequence $\varepsilon_n\to0$ and $\u_{0,n}\to\u_0$ in $\H$. By Corollary \ref{Perturbation_u} we find that for every $t\geq0, \tau\in\R$ and $\omega\in\Omega$,
		\begin{align}\label{U-SC2}
			\Phi_{\varepsilon}(t,\tau,\omega,\u_{0,n}) \to \Phi_0(t;\tau,\u_{0}) \ \text{ in } \ \H.
		\end{align}
		Hence \eqref{K1}, \eqref{U-SC1}, \eqref{U-SC2} and Lemma \ref{precompact} along with Theorem 3.2 in \cite{non-autoUpperWang} complete the proof.
	\end{proof}

	\begin{appendix}
		\renewcommand{\thesection}{\Alph{section}}
		\numberwithin{equation}{section}
		\section{Uniform tail estimate for the solution of SCBF equations}\label{sec7}\setcounter{equation}{0}
		One can obtain the results obtained in the previous sections using 	the uniform tail estimate for the solution of SCBF equations also. We omitted this method, since we have a restriction on $r\geq 2$ in two dimensions and we need $\f\in\mathrm{L}^2_{\mathrm{loc}}(\R;\H)$.  To prove the uniform tail estimate for the solution of the system \eqref{SCBF}, we assume that the external forcing term $\f\in\mathrm{L}^2_{\mathrm{loc}}(\R;\H)$ and there exists a number $\delta\in[0,\alpha)$ such that
		\begin{align}\label{forcing3}
			\int_{-\infty}^{\tau} e^{\delta\xi}\|\f(\cdot,\xi)\|^2_{\H}\d \xi<\infty, \ \text{ for all } \ \tau\in\R.
		\end{align}
		Moreover, \eqref{forcing3} implies that 
		\begin{align}\label{forcing4}
			\lim_{k\to\infty}\int_{-\infty}^{0}\int\limits_{|x|\geq k} e^{\delta\xi}|\f(x,\xi+\tau)|^2\d x\d \xi=0, \ \text{ for all }\ \tau\in\R.
		\end{align}
		\begin{lemma}\label{largeradius}
			For $0<\varepsilon\leq 1, d=2$ with $r\geq2, \ d=3$ with $r>3$ and $d=r=3$ with $2\beta\mu\geq1$, assume that $\f\in\mathrm{L}^2_{\mathrm{loc}}(\R;\H)$ and satisfies \eqref{forcing3}. Then, for any $\v_{\varepsilon,\tau-t}\in D(\tau-t,\vartheta_{-t}\omega),$ where $D=\{D(\tau,\omega):\tau\in\R, \omega\in \Omega\}\in\mathfrak{D}$, and for any $\eta>0$ and $\mathbb{P}$-a.e. $\omega\in \Omega$, there exist $\mathscr{T}^*=\mathscr{T}^*(\tau,\omega,D,\eta)\geq 1$ and $P^*=P^*(\tau,\omega,\eta)>0$ such that for all $t\geq \mathscr{T}^*$, the solution of \eqref{CCBF}  with $\omega$ replaced by $\vartheta_{-\tau}\omega$ satisfy
			\begin{align}\label{ep}
				\int_{|x|\geq P^*}|\v_{\varepsilon}(\tau,\tau-t,\vartheta_{-\tau}\omega,\v_{\varepsilon,\tau-t}) |^2\d x\leq \eta,
			\end{align}
			where $\mathscr{T}^*(\tau,\omega,D,\eta)$ and $P^*(\tau,\omega,\eta)$ are independent of $\varepsilon$.
		\end{lemma}
		\begin{proof}
			Let $\upxi$ be a smooth function such that $0\leq\upxi(s)\leq 1$ for $s\in\R^+$ and 
			\begin{align}\label{xi}
				\upxi(s)=\begin{cases*}
					0,\quad \text{ for }0\leq s\leq 1,\\
					1, \quad\text{ for } s\geq2 .
				\end{cases*}
			\end{align}
			Then, there exists a positive constant $C$ such that $|\upxi^{'}(s)|\leq C$ for all $s\in\R^+$. Taking the inner product of first equation of \eqref{CCBF} with $\upxi\left(\frac{|x|^2}{k^2}\right)\v_{\varepsilon}$ in $\H$, we have
			\begin{align}\label{ep1}
				&\frac{1}{2} \frac{\d}{\d t} \int_{\R^d}\upxi\left(\frac{|x|^2}{k^2}\right)|\v_{\varepsilon}|^2\d x \nonumber\\&= -\mu \int_{\R^d}(\A\v_{\varepsilon}) \upxi\left(\frac{|x|^2}{k^2}\right) \v_{\varepsilon} \d x-\alpha \int_{\R^d}\upxi\left(\frac{|x|^2}{k^2}\right)|\v_{\varepsilon}|^2\d x-\frac{1}{\z(t,\omega)}b(\v_{\varepsilon},\v_{\varepsilon},\upxi\left(\frac{|x|^2}{k^2}\right)\v_{\varepsilon})\nonumber\\&\quad-\frac{\beta}{[\z(t,\omega)]^{r-1}} \int_{\R^d}\upxi\left(\frac{|x|^2}{k^2}\right)|\v_{\varepsilon}|^{r+1}\d x+\z(t,\omega) \int_{\R^d}\f\upxi\left(\frac{|x|^2}{k^2}\right)\v_{\varepsilon}\d x.
			\end{align}
			Now, we estimate each term on the right hand side of \eqref{ep1}. Applying the integration by parts, we obtain
			\begin{align}\label{ep2}
				-&\mu \int_{\R^d}(\A\v_{\varepsilon}) \upxi\left(\frac{|x|^2}{k^2}\right) \v_{\varepsilon} \d x= -\mu \int_{\R^d}|\nabla\v_{\varepsilon}|^2 \upxi\left(\frac{|x|^2}{k^2}\right)  \d x -\mu \int_{\R^d}\v_{\varepsilon} \upxi^{'}\left(\frac{|x|^2}{k^2}\right)\frac{2x}{k^2}\cdot\nabla \v_{\varepsilon} \d x,
			\end{align}
			and
			\begin{align}\label{ep2*}
				-\mu \int_{\R^d}\v_{\varepsilon} \upxi^{'}\left(\frac{|x|^2}{k^2}\right)\frac{2x}{k^2}\cdot\nabla \v_{\varepsilon} \d x&=  -\mu \int\limits_{k\leq|x|\leq \sqrt{2}k}\v_{\varepsilon} \upxi^{'}\left(\frac{|x|^2}{k^2}\right)\frac{2x}{k^2}\cdot\nabla \v_{\varepsilon} \d x\nonumber\\&\leq  \frac{2\sqrt{2}\mu}{k} \int\limits_{k\leq|x|\leq \sqrt{2}k}\left|\v_{\varepsilon}\right| \left|\upxi^{'}\left(\frac{|x|^2}{k^2}\right)\right|\left|\nabla \v_{\varepsilon}\right| \d x\nonumber\\&\leq \frac{C}{k} \int_{\R^d}\left|\v_{\varepsilon}\right| \left|\nabla \v_{\varepsilon}\right| \d x \leq \frac{C}{k} \left(\|\v_{\varepsilon}\|^2_{\H}+\|\nabla\v_{\varepsilon}\|^2_{\H}\right).
			\end{align}
			Also, for $r\geq2$, we obtain
			\begin{align}\label{ep3}
				-\frac{1}{\z(t,\omega)}b(\v_{\varepsilon},\v_{\varepsilon},\upxi\left(\frac{|x|^2}{k^2}\right)\v_{\varepsilon})&=\frac{1}{\z(t,\omega)}\int_{\R^d} \upxi^{'}\left(\frac{|x|^2}{k^2}\right)\frac{x}{k^2}\cdot\v_{\varepsilon} |\v_{\varepsilon}|^2 \d x\nonumber\\&= \frac{1}{\z(t,\omega)} \int\limits_{k\leq|x|\leq \sqrt{2}k} \upxi^{'}\left(\frac{|x|^2}{k^2}\right)\frac{x}{k^2}\cdot\v_{\varepsilon} |\v_{\varepsilon}|^2 \d x\nonumber\\&\leq \frac{1}{\z(t,\omega)}\frac{\sqrt{2}}{k} \int\limits_{k\leq|x|\leq \sqrt{2}k} \left|\upxi^{'}\left(\frac{|x|^2}{k^2}\right)\right| |\v_{\varepsilon}|^3 \d x\nonumber\\&\leq \frac{1}{\z(t,\omega)}\frac{C}{k}\|\v_{\varepsilon}\|^3_{\wi \L^3}\leq \frac{1}{\z(t,\omega)}\frac{C}{k}\|\v_{\varepsilon}\|^{\frac{2(r-2)}{r-1}}_{\H}\|\v_{\varepsilon}\|^{\frac{r+1}{r-1}}_{\wi \L^{r+1}}\nonumber\\&\leq \frac{C}{k}\left(\|\v_{\varepsilon}\|^2_{\H}+\frac{1}{[\z(t,\omega)]^{r-1}}\|\v_{\varepsilon}\|^{r+1}_{\wi \L^{r+1}}\right),
			\end{align}
			where we have used interpolation (Lemma \ref{Interpolation}) and Young's inequalities. Now, we estimate the last term of \eqref{ep1} as follows
			\begin{align}\label{ep4}
				&	\z(t,\omega) \int_{\R^d}\f\upxi\left(\frac{|x|^2}{k^2}\right)\v_{\varepsilon} \d x\leq \frac{\alpha}{2} \int_{\R^d}\upxi\left(\frac{|x|^2}{k^2}\right)|\v_{\varepsilon}|^2\d x +\frac{[\z(t,\omega)]^2}{2\alpha} \int_{\R^d}\upxi\left(\frac{|x|^2}{k^2}\right)|\f|^2\d x.
			\end{align}  
			Making use of \eqref{ep2}-\eqref{ep4} in \eqref{ep1}, we get
			\begin{align}\label{ep5}
				\frac{\d}{\d t} \int_{\R^d}\upxi\left(\frac{|x|^2}{k^2}\right)|\v_{\varepsilon}|^2\d x &\leq -\alpha \int_{\R^d}\upxi\left(\frac{|x|^2}{k^2}\right)|\v_{\varepsilon}|^2\d x+\frac{C}{k} \left(\|\v_{\varepsilon}\|^2_{\H}+\|\nabla\v_{\varepsilon}\|^2_{\H}\right)\nonumber\\&\quad+\frac{C}{k[\z(t,\omega)]^{r-1}}\|\v_{\varepsilon}\|^{r+1}_{\wi \L^{r+1}}+\frac{[\z(t,\omega)]^2}{\alpha} \int_{\R^d}\upxi\left(\frac{|x|^2}{k^2}\right)|\f|^2\d x.
			\end{align}
			Making use of Gronwall's inequality to the above equation \eqref{ep5} on $(\tau-t,\tau)$ and replacing $\omega$ by $\vartheta_{-\tau}\omega$, we find that, for $\tau\in\R, t\geq 0$ and $\omega\in \Omega$,
			\begin{align}\label{ep6}
				&\int_{\R^d}\upxi\left(\frac{|x|^2}{k^2}\right)|\v_{\varepsilon}(\tau,\tau-t,\vartheta_{-\tau}\omega,\v_{\varepsilon,\tau-t})|^2\d x \leq e^{-\alpha t}\int_{\R^d}\upxi\left(\frac{|x|^2}{k^2}\right)|\v_{\varepsilon,\tau-t}|^2\d x\nonumber\\&\quad+\frac{C}{k}\int_{\tau-t}^{\tau}e^{\alpha(s-\tau)} \left(\|\v_{\varepsilon}(s,\tau-t,\vartheta_{-\tau}\omega,\v_{\varepsilon,\tau-t})\|^2_{\H}+\|\nabla\v_{\varepsilon}(s,\tau-t,\vartheta_{-\tau}\omega,\v_{\varepsilon,\tau-t})\|^2_{\H}\right)\d s\nonumber\\&\quad+\frac{C}{k}\int_{\tau-t}^{\tau}\frac{e^{\alpha (s-\tau)}}{[\z(s,\vartheta_{-\tau}\omega)]^{r-1}}\|\v_{\varepsilon}(s,\tau-t,\vartheta_{-\tau}\omega,\v_{\varepsilon,\tau-t})\|^{r+1}_{\wi\L^{r+1}}\d s\nonumber\\&\quad+\frac{1}{\alpha}\int_{\tau-t}^{\tau}e^{\alpha (s-\tau)}[\z(s,\vartheta_{-\tau}\omega)]^2 \int_{\R^d}\upxi\left(\frac{|x|^2}{k^2}\right)|\f|^2\d x\d s.
			\end{align}
			From \eqref{initial_data}, it easy to obtain that, given $\eta>0,$ there exists $T_1=T_1(\tau,\omega,D,\eta)\geq 1$, independent of $\varepsilon$, such that  for all $t\geq T_1,$
			\begin{align}\label{ep7}
				e^{-\alpha t}\int_{\R^d}\upxi\left(\frac{|x|^2}{k^2}\right)|\v_{\varepsilon,\tau-t}|^2\d x\leq\frac{\eta}{5}.
			\end{align}
			Further, using \eqref{ue}, we get that
			\begin{align}\label{ep8}
				&\frac{C}{k}\int_{\tau-t}^{\tau}e^{\alpha (s-\tau)} \|\v_{\varepsilon}(s,\tau-t,\vartheta_{-\tau}\omega,\v_{\varepsilon,\tau-t})\|^2_{\H}\d s\nonumber\\&\leq\frac{C}{k}te^{-\alpha t} \|\v_{\varepsilon,\tau-t}\|^2_{\H}+ \frac{C}{k}\int_{\tau-t}^{\tau} \int_{\tau-t}^{s}e^{\alpha\xi} [\z(\xi,\vartheta_{-\tau}\omega)]^2 \|\f(\cdot,\xi)\|^2_{\H}\d \xi\d s\nonumber\\&\leq\frac{C}{k}te^{-\alpha t} \|\v_{\varepsilon,\tau-t}\|^2_{\H}+ \frac{C}{k}\int_{\tau-t}^{\tau} \int_{-t}^{s-\tau}e^{\alpha\xi} [\z(\xi+\tau,\vartheta_{-\tau}\omega)]^2 \|\f(\cdot,\xi+\tau)\|^2_{\H}\d \xi\d s\nonumber\\&\leq\frac{C}{k}te^{-\alpha t} \|\v_{\varepsilon,\tau-t}\|^2_{\H}+ \frac{C}{ k}\int_{\tau-t}^{\tau}e^{\frac{\alpha}{2}(s-\tau)} \int_{-t}^{s-\tau}e^{\frac{\alpha}{2}\xi}[\z(\xi+\tau,\vartheta_{-\tau}\omega)]^2\|\f(\cdot,\xi+\tau)\|^2_{\H}\d \xi\d s\nonumber\\&\leq\frac{C}{k}te^{-\alpha t} \|\v_{\varepsilon,\tau-t}\|^2_{\H}+ \frac{C}{ k}\int_{\tau-t}^{\tau}e^{\frac{\alpha}{2}(s-\tau)} \d s\int_{-\infty}^{0}e^{\frac{\alpha}{2}\xi}[\z(\xi+\tau,\vartheta_{-\tau}\omega)]^2\|\f(\cdot,\xi+\tau)\|^2_{\H}\d \xi\nonumber\\&\leq\frac{C}{k}te^{-\alpha t} \|\v_{\varepsilon,\tau-t}\|^2_{\H}+ \frac{C}{ k}\int_{-\infty}^{0}e^{\frac{\alpha}{2}\xi}[\z(\xi+\tau,\vartheta_{-\tau}\omega)]^2\|\f(\cdot,\xi+\tau)\|^2_{\H}\d \xi.
			\end{align}
			Since $\f\in\mathrm{L}^2_{\mathrm{loc}}(\R;\H)$ satisfying \eqref{forcing1} and $\v_{\varepsilon,\tau-t}\in D(\tau-t,\vartheta_{-t}\omega)$, there exists $T_2=T_2(\tau,\omega,D,\eta)>T_1$ and $P_1=P_1(\tau,\omega,\eta)>0,$ independent of $\varepsilon$, such that for all $t\geq T_2$ and $k\geq P_1$,
			\begin{align}\label{ep9}
				\frac{C}{k}\int_{\tau-t}^{\tau}e^{\alpha (s-\tau)} \|\v_{\varepsilon}(s,\tau-t,\vartheta_{-\tau}\omega,\v_{\varepsilon,\tau-t})\|^2_{\H}\d s\leq \frac{\eta}{5}.
			\end{align}
			From \eqref{ue^1}, we see that there exists $T_3=T_3(\tau,\omega,D,\eta)>T_1$ and $P_2=P_2(\tau,\omega,\eta)>0,$ independent of $\varepsilon$, such that for all $t\geq T_3$ and $k\geq P_2$,
			\begin{align}\label{ep10}
				\frac{C}{k}\int_{\tau-t}^{\tau}e^{\alpha (s-\tau)} \|\nabla\v_{\varepsilon}(s,\tau-t,\vartheta_{-\tau}\omega,\v_{\varepsilon,\tau-t})\|^2_{\H}\d s\leq \frac{\eta}{5}.
			\end{align}
			We find from \eqref{ue^2} that there exists $T_4=T_4(\tau,\omega,D,\eta)>T_1$ and $P_3=P_3(\tau,\omega,\eta)>0,$ independent of $\varepsilon$, such that for all $t\geq T_4$ and $k\geq P_3$,
			\begin{align}\label{ep11}
				\frac{C}{k}\int_{\tau-t}^{\tau}\frac{e^{\alpha (s-\tau)}}{[\z(s,\vartheta_{-\tau}\omega)]^{r-1}}\|\v_{\varepsilon}(s,\tau-t,\vartheta_{-\tau}\omega,\v_{\varepsilon,\tau-t})\|^{r+1}_{\wi\L^{r+1}}\d s\leq \frac{\eta}{5}.
			\end{align}
			We also get that there exists $T_5=T_5(\tau,\omega,D,\eta)>T_1$ and $P_4=P_4(\tau,\omega,\eta)>0,$ independent of $\varepsilon$, such that for all $t\geq T_5$ and $k\geq P_4$,
			\begin{align}\label{ep12}
				&\frac{1}{\alpha}\int_{\tau-t}^{\tau}e^{\alpha (s-\tau)}[\z(s,\vartheta_{-\tau}\omega)]^2 \int_{\R^d}\upxi\left(\frac{|x|^2}{k^2}\right)|\f(x,s)|^2\d x\d s\nonumber\\&\leq\frac{1}{\alpha}\int_{\tau-t}^{\tau}e^{\alpha (s-\tau)}[\z(s,\vartheta_{-\tau}\omega)]^2 \int\limits_{|x|\geq k}|\f(x,s)|^2\d x\d s\nonumber\\&\leq\frac{1}{\alpha}\int_{-\infty}^{0}\int\limits_{|x|\geq k}e^{\alpha s}[\z(s+\tau,\vartheta_{-\tau}\omega)]^2 |\f(x,s+\tau)|^2\d x\d s\leq \frac{\eta}{5},
			\end{align}
			where we have used the fact that $\f\in\mathrm{L}^2_{\mathrm{loc}}(\R;\H)$ and satisfying \eqref{forcing4}. 
			
			Let $\mathscr{T}^*=\mathscr{T}^*(\tau,\omega,D,\eta)=\max\{T_1,T_2,T_3,T_4,T_5\}, P^*=P^*(\tau,\omega,\eta)=\max\{P_1,P_2,P_3,P_4\}$, then it follows from \eqref{ep7}-\eqref{ep12} that, for all $t\geq \mathscr{T}^*$ and $k\geq P^*$, we obtain
			\begin{align*}
				\int_{\R^d}\upxi\left(\frac{|x|^2}{k^2}\right)|\v_{\varepsilon}(\tau,\tau-t,\vartheta_{-\tau}\omega,\v_{\varepsilon,\tau-t})|^2\d x\leq \eta,
			\end{align*}
			which implies \eqref{ep}.
		\end{proof}
		
	\end{appendix}

	\medskip\noindent
	{\bf Acknowledgments:}    The first author would like to thank the Council of Scientific $\&$ Industrial Research (CSIR), India for financial assistance (File No. 09/143(0938)/2019-EMR-I).  M. T. Mohan would  like to thank the Department of Science and Technology (DST), Govt of India for Innovation in Science Pursuit for Inspired Research (INSPIRE) Faculty Award (IFA17-MA110).

\end{document}